\documentclass{amsart}
\usepackage{amsmath,amssymb}
\newtheorem{theorem}{Theorem}[section]
\newtheorem{proposition}[theorem]{Proposition}
\newtheorem{corollary}[theorem]{Corollary}

\newtheorem{lemma}[theorem]{Lemma}

\theoremstyle{definition}
\newtheorem{definition}[theorem]{Definition}
\newtheorem{example}[theorem]{Example}

\theoremstyle{remark}
\newtheorem{remark}[theorem]{Remark}

\numberwithin{equation}{section}

\newcommand{\al}{\alpha}
\newcommand{\be}{\beta}
\newcommand{\de}{\delta}
\newcommand{\ep}{\epsilon}

\newcommand{\ga}{\gamma}

\newcommand{\la}{\lambda}
\newcommand{\om}{\omega}
\newcommand{\si}{\sigma}

\newcommand{\vp}{\varphi}

\newcommand{\De}{\Delta}

\newcommand{\La}{\Lambda}
\newcommand{\Si}{\Sigma}
\newcommand{\Om}{\Omega}
\def\CC{\mathbb{C}}
\def\NN{\mathbb{N}}
\def\RR{\mathbb{R}}

\def\ZZ{\mathbb{Z}}

\renewcommand\SS{\mathbb{S}}
\newcommand{\cC}{{\mathcal C}}

\newcommand{\cE}{{\mathcal E}}

\newcommand{\cG}{{\mathcal G}}

\newcommand{\cL}{{\mathcal L}}

\newcommand{\cO}{{\mathcal O}}

\newcommand{\cV}{{\mathcal V}}

\newcommand{\pd}{\partial}
\newcommand\minus\backslash
\newcommand{\id}{{\rm id}}

\newcommand{\ms}{\mspace{1mu}}
\newcommand\lan\langle
\newcommand\ran\rangle

\newcommand{\inte}{\operatorname{int}}

\newcommand{\e}{{\mathrm e}}
\newcommand{\dd}{{\mathrm d}}
\DeclareMathOperator\Div{div} 
\DeclareMathOperator\Rm{Rm} 
 \DeclareMathOperator\Isom{Isom}

\DeclareMathOperator\diag{diag} 
\DeclareMathOperator\dist{dist} 
\DeclareMathOperator\Cr{Cr}

\renewcommand\leq\leqslant
\renewcommand\geq\geqslant
\newlength{\intwidth}

\newcommand\BOm{\overline\Om}
\DeclareMathOperator\Fix{Fix}

 \DeclareMathOperator\ind{ind}

\DeclareMathOperator\SO{SO}
\begin{document}

\title[Critical points of Green's functions]{Critical points of Green's functions\\ on complete manifolds}

\author{Alberto Enciso}
\address{Departement Mathematik, ETH Z\"urich, 8092 Z\"urich, Switzerland}
\email{alberto.enciso@math.ethz.ch}

\author{Daniel Peralta-Salas}
\address{Instituto de Ciencias Matem\'aticas, CSIC-UAM-UC3M-UCM, C/ Serrano 123, 28006 Madrid, Spain}
\email{dperalta@icmat.es}

%
%
\begin{abstract}
We prove that the number of critical points of a Li--Tam Green's function on a complete open Riemannian surface of finite type admits a topological upper bound, given by the first Betti number of the surface. In higher dimensions, we show that there are no topological upper bounds on the number of critical points by constructing, for each nonnegative integer $N$, a manifold diffeomorphic to $\RR^n$ ($n\geq3$) whose minimal Green's function has at least $N$ nondegenerate critical points. Variations on the method of proof of the latter result yield contractible  $n$-manifolds whose minimal Green's functions have level sets diffeomorphic to any fixed codimension $1$ compact submanifold of $\RR^n$.
\end{abstract}
\maketitle

\section{Introduction}
\label{S:intro}

Let $(M,g)$ be a noncompact, complete Riemannian
{$n$-manifold} without boundary and let us denote
by $\cG: (M\times M)\minus\diag(M\times M)\to\RR$ a symmetric Green's function of $(M,g)$, which satisfies
\begin{equation}\label{Gr}
\De_g\cG(\cdot,y)=-\de_y
\end{equation}
for each $y\in M$. We will find it notationally convenient to fix a point $y\in M$, once and for all, and consider a Green's function $G:=\cG(\cdot,y)$ with pole $y$, which is smooth and harmonic in $M\minus\{y\}$.

The study of the Green's functions of the Laplacian in a complete
Riemannian manifold is a classical problem in geometric analysis and
partial differential equations. Consequently, there is a vast
literature on this topic covering,
among many other aspects, the
existence of positive Green's functions~\cite{CY75,LT87,LT92}, upper and
lower bounds, gradient estimates and
asymptotics~\cite{LY86,LT95,CM97,Ho99}, and the connection between Green's
functions and the heat kernel~\cite{Va82,LTW97,GS02}.

In this paper we shall focus on the study of the critical points of Green's functions on a complete Riemannian manifold. The chief difficulty lies in the fact that, generally speaking, the Green's
function estimates are not sufficiently fine to elucidate whether the
gradient of $G$ vanishes in a certain region. Moreover, it is well
known that the codimension of the critical set of $G$ is at least~2
\cite{HS89}, which introduces additional complications in the
analysis. For this reason, our approach is based on a combination of techniques from the geometric theory of dynamical systems, transversality theory and second-order elliptic PDEs.

Let us state our main results. The first theorem asserts that there is a topological upper bound for the number of critical points of any Li--Tam Green's function on a surface of finite type. (The notion of Li--Tam Green's function, which generalizes that of minimal Green's function, is recalled in Section~\ref{S:critical}.)

\begin{theorem}\label{T:dim2}
Let $(M,g)$ be a smooth open Riemannian surface of finite type. Then the number of critical points of any Li--Tam Green's function $G$ on $M$ is not larger than the first Betti number $b_1(M)$, and this upper bound is attained if and only if $G$ is Morse.
\end{theorem}

Theorem~\ref{T:dim2} is a substantial extension of the classical result~\cite{Wa50} that the Dirichlet Green's function of a simply or doubly connected domain in the Euclidean plane respectively has zero or one critical points and, to our best knowledge, is the first general finiteness result for critical points in noncompact manifolds. The crucial step in the proof of Theorem~\ref{T:dim2} is to show that the number of critical points of $G$ is necessarily finite, which requires a delicate local and global analysis of the possible saddle connections. It is not clear to us how to prove the finiteness of the critical set using complex function theory or PDE methods, even for analytic metrics.

Our second result complements~\ref{T:dim2} by showing that the number of critical points of the Green's function of a manifold of dimension $n\geq3$ cannot admit a topological upper bound. In fact,  we provide a procedure for constructing analytic metrics in $\RR^n$ ($n\geq 3$) whose minimal Green's functions have level sets of prescribed topology and any finite number of nondegenerate critical points, which allows us to prove the following

\begin{theorem}\label{T:everything}
  Let $N$ be a positive integer and $\Si$ a smooth codimension $1$ closed submanifold of $\RR^n$. For any $n\geq3$ there exist real analytic complete Riemannian manifolds $(M_j,g_j)$ ($1\leq j\leq 3$) diffeomorphic to $\RR^n$ such that:
\begin{enumerate}
\item The minimal Green's function of $(M_1,g_1)$ has at least $N$ nondegenerate critical points.
\item The minimal Green's function of $(M_2,g_2)$ has a level set diffeomorphic to~$\Si$.
\item The critical set of the minimal Green's function of $(M_3,g_3)$ has codimension at most $3$.
\end{enumerate}
\end{theorem}

Additional motivation for this theorem comes from a question of Kawohl~\cite{Kaw88}, recently solved in~\cite{EP08}, concerning the possible level sets and critical points of the solution to an exterior boundary problem in Euclidean space. In the context of Riemannian
geometry, the natural analogue of Kawohl's problem is whether there
are nontrivial restrictions on the critical and level sets of the
minimal Green's function (when it exists) of a Riemannian manifold
diffeomorphic to $\RR^n$. The first statement in Theorem~\ref{T:everything} is 
reminiscent of results of Morse and Sheldon~\cite{Mo70,Sh80} on the existence of Morse harmonic functions in bounded domains of
$\RR^2$ and $\RR^3$ with an arbitrary number of nondegenerate critical
points, but these authors' constructions cannot be modified to deal with the problem studied in this paper.

The paper is organized as follows. In Section~\ref{S:critical} we derive some preliminary results on the integral curves of the gradient of $G$ that will be of use in the rest of the paper. In Section~\ref{S:surfaces} we give the proof of Theorem~\ref{T:dim2} and, using the same ideas, show that a Li--Tam Green's function of a manifold diffeomorphic to $\RR^n$ with an $\SO(n-1)$ isometry group does not have any critical points. In Section~\ref{S:Morse} we prove that the Dirichlet Green's function of a bounded domain in a Riemannian manifold is generically Morse, a result we need for the proof of Theorem~\ref{T:everything} and which cannot be obtained from Albert's, Uhlenbeck's or Bando--Urakawa's analogous theorems for the eigenfunctions of the Laplacian~\cite{Uh76,Al78,BU83} or Damon's results for filtered differential operators~\cite{Da97}. Finally, in Section~\ref{S:main} we present the proof of Theorem~\ref{T:everything}.

We conclude this section by introducing some standard notation. We shall denote by $\De_g$, $\nabla_g$,
$\dd V_g$, $|\cdot|_g$ and $\dist_g$, respectively, the Laplacian,
gradient operator, Riemannian measure, norm and distance function in
$(M,g)$, and we shall reserve the notation $\De$, $\nabla$, $\dd x$,
$|\cdot|$ and $\dist$ for the corresponding objects in Euclidean space
$(\RR^n,g_0)$. We shall use the notation  $B_g(x,r)$
and $B(x,r)$, in each case, for the open geodesic balls in $(M,g)$ and in $(\RR^n,g_0)$
of center $x$ and radius $r$. The {\em critical set} of a $C^1$ function $f$ will be denoted by $\Cr(f)$, and is defined as the set of points $x$ in the domain of $f$ such that $\dd f(x)=0$. Throughout this paper, the manifold $(M,g)$ will be assumed to be connected, oriented and of class $C^\infty$.

\section{Preliminary results on critical points and level sets of Green's functions}
\label{S:critical}

In this section we shall derive some preliminary results on the
level sets and critical points of $G$ in a complete manifold $(M,g)$
of dimension $n\geq 3$ using techniques from the geometric theory of
dynamical systems. The approach and results presented
in this section are based, and significantly extend, Brelot and
Choquet's work on Green's lines in classical potential
theory~\cite{BC51}, and will be of much use in forthcoming
sections. The two-dimensional case presents peculiarities of its own,
and will be treated in detail in Section~\ref{S:surfaces}. As the
structure of a function in a neighborhood of a critical point depends
heavily on whether we are in the smooth or analytic category, we shall
assume throughout this section that $(M,g)$ is a $C^\om$ Riemannian
manifold; in particular, this precludes some of the subtleties present
in the case of general $C^\infty$
metrics~\cite{DoFe88,HS89,HHN99}.

It is customary to impose some decay conditions at infinity in order
to control certain global properties of the level sets of $G$; in
particular, a standard choice is to restrict one's attention to the
minimal Green's function whenever it exists. More generally, we shall
assume in what follows that the Green's function $G$ has been obtained
through Li and Tam's exhaustion procedure~\cite{LT87}, which provides
some control on the behavior of $G$ at infinity. This kind of
solutions to Eq.~\eqref{Gr} will be called {\em Li--Tam Green's
  functions}. Whereas this requirement does not have any bearing on
the local results below, it is essential in the analysis of the global
properties of $G$ that we shall also present in this section.

Let us henceforth denote by $\cG_\Om:(\Om\times\Om)\minus\diag(\Om\times\Om)\to\RR$ the symmetric Dirichlet Green's function of a bounded domain $\Om\subset M$, which is defined by
\begin{equation}\label{GrDom}
\De_g\cG_\Om(\cdot,y)=-\de_y\quad\text{in }\Om\,,\qquad \cG_\Om(\cdot,y)=0\quad \text{on
}\pd\Om\,,
\end{equation}
and set $G_\Om:=\cG_\Om(\cdot,y)$. For completeness, we recall that a Li--Tam Green's function
is constructed from an exhaustion $\Om_1\subset\Om_2\subset\cdots$ of
$M$ by bounded domains. A theorem of Li and Tam~\cite{LT87}
ensures that in any smooth Riemannian manifold there exists a sequence of nonnegative real numbers
$(a_j)_{j=1}^\infty$ such that $G_{\Om_j}-a_j$ converges uniformly on
compact sets of $M\minus\{y\}$ to a Green's function $G$ with pole $y$, and
that it coincides with the minimal one whenever the latter
exists. Li--Tam Green's functions are generally nonunique, but in any case a Li--Tam Green's function $G$ possesses the
following properties~\cite{LT87}:
\begin{enumerate}
\item $G$ is decreasing, i.e.,
\[
\sup_{M\minus B_g(y,r)} G=\max_{\pd B_g(y,r)}G
\]
for all $r>0$, $B_g(y,r)$ being the geodesic ball centered at the pole $y$
of radius $r$.

\item $G$ tends to a definite limit (possibly $-\infty$) at
each end of $M$.

\item If $H$ is an amenable isometry group of $(M,g)$, one can assume that $\cG(\phi(x),\phi(y))=\cG(x,y)$ for all $x,y\in M$ and $\phi\in H$ (cf.\ e.g.~\cite{EP07}).
\end{enumerate}

First of all, let us recall some definitions from real analytic geometry. The {\em dimension} of a closed analytic set $S$ is
\[
\dim(S)=\sup_{x\in S}\dim_x(S)\,,
\]
the {\em local dimension} of $S$ at $x$ being
\[
\dim_x(S)=\sup\big\{\dim(P):P\text{ open topological submanifold of }S,\, x\in P\big\}\,.
\]
This is simply the dimension of the top-dimensional strata of the Lojasiewicz stratification of $S$ at $x$~\cite{BM88}. $S$ is said to have {\em pure dimension} $k$ if $\dim_x(S)=k$ for all $x\in S$.

In the first place, in the following elementary proposition we shall collect some easy
properties of the critical set of $G$ which will be frequently used in
the rest of the section. 

\begin{proposition}\label{P:elem}
Both $\Cr(G)$ and the level sets $G^{-1}(c)$ ($c>\inf G$)
are analytic sets, and their connected components are compact and without
boundary. Moreover, $G^{-1}(c)$ has pure codimension $1$ and $\Cr(G)$ has codimension at
least~$2$.
\end{proposition}
\begin{proof}
  $\Cr(G)$ and $G^{-1}(c)$ are certainly analytic
  because so is $G$, and therefore their connected components have no
  boundary by a theorem of Sullivan~\cite{Su71}. Each component of a level set of $G$ is compact
  because $G$ is decreasing and tends to a definite limit at each end
  of the manifold. It is well known~\cite{HS89} that the codimension of $\Cr(G)$ is at least $2$ (in the analytic case, this property follows directly from the Lojasiewicz structure theorem and Cauchy--Kowalewski theorem). 

The analyticity of $G$ in $M\backslash \{y\}$ implies that $G^{-1}(c)$ admits a Lojasiewicz stratification~\cite[Theorem 6.3.3]{KP02} and has codimension at least 1. In order to prove that $G^{-1}(c)$ has pure codimension $1$, let us suppose that
this is not the case. Then there must exist a
point $x$ and a neighborhood $U\ni x$ such that the connected set
$V:=G^{-1}(c)\cap U$ has pure codimension $k\geq2$~\cite{BM88,KP02}. By the implicit
function theorem, $V$ is contained in the critical set of $G$, and
Lojasiewicz's vanishing theorem~\cite[Theorem 6.3.4]{KP02} shows that
the neighboring level sets of $G$ in $U$ are tubes around $V$. Hence $V$
is a local maximum or minimum of $G$, contradicting its harmonicity.
\end{proof}
\begin{remark}\label{RP:elem}
Proposition~\ref{P:elem} and its proof apply to Dirichlet Green's functions as well. If $f$ is a harmonic function, it also follows from the proof of the proposition that $f^{-1}(c)$ has pure codimension $1$ and that the codimension of its critical set is at least $2$.
\end{remark}

%

As is well known, critical level sets are not necessarily
submanifolds, so that their local structure can be rather
complicated. We shall present next a technical result on the structure
of a critical level set in a neighborhood of an isolated critical
point which is crucial for the proof of
Proposition~\ref{P:halfline}. This proposition allows to rule out,
e.g., the appearance of certain kind of $n$-dimensional cusp
singularities~\cite{Ar85}. The proof makes use of a lemma based on
ideas by Brelot and Choquet~\cite{BC55}, a simple proof thereof is
given below.

\begin{lemma}\label{L:HarmPoly}
Let $Q,R$ be real homogeneous polynomials in $n$ variables. Let us
suppose that $P(x):=Q(x)R(x)$ satisfies $\De P=0$ and that $Q^{-1}(0)\subset R^{-1}(0)$. Then $Q$ is
constant.
\end{lemma}
\begin{proof}
Let us call $\SS^{n-1}$ the unit sphere in $\RR^n$ and denote by a
bar the restriction of a polynomial in $\RR^n$ to $\SS^{n-1}$. As
$P$ is harmonic and homogenous of degree $d:=\deg P$, it follows
that $\overline P$ is a spherical harmonic of order $d$. If $Q$
is not constant, $R$ has degree at most $d-1$, so there exist
spherical harmonics $Y_j$ of order $j$ and constants $c_j$ such
that
\[
\overline R=\sum_{j=0}^{d-1}c_jY_j\,.
\]
Hence it follows that $\overline R$ and $\overline P$ are orthogonal, i.e.,
\begin{equation}\label{int0}
\int_{\SS^{n-1}}\overline R\,\overline P\,\dd\si=\int_{\SS^{n-1}}\overline
R^2\,\overline Q\,\dd\si=0\,.
\end{equation}
Let us now observe that $P^{-1}(0)=R^{-1}(0)$ has pure codimension 1 by Proposition~\ref{P:elem} and Remark~\ref{RP:elem}.
It then follows that $Q^{-1}(0)$ has codimension $1$ in $R^{-1}(0)$, since
otherwise $\nabla P=Q\,\nabla R+R\,\nabla Q$ would vanish identically on a
codimension 1 set, contradicting the fact that
$\Cr(P)$ has codimension at least $2$. In turn, this implies that $Q$ does not
change sign, in contradiction with the fact that the
integral~\eqref{int0} is zero.
\end{proof}

\begin{proposition}\label{T:regions}
Let $U$ be a small neighborhood of an isolated critical point $z$
of $G$ and set $c:=G(z)$. If $z$ is also an isolated critical point of the first nonzero homogeneous term of the Taylor expansion of $G$ at $z$, then $U\minus G^{-1}(c)$ has at least three
connected components.
\end{proposition}
\begin{proof}
By Proposition~\ref{P:elem}, $G^{-1}(c)$ is a pure codimension $1$
set without boundary, and thus disconnects $U$. Now suppose that
$U\minus G^{-1}(c)$ has only two components. By Lojasiewicz's structure
theorem~\cite[Section 6.3]{KP02}, $\La:=G^{-1}(c)\cap U$ must be homeomorphic
to a hyperplane, since otherwise $U\minus\La$ would have more than two
connected components.

In what follows we take an analytic chart $x=(x_1,\dots, x_n)$,
thereby embedding $U$ in $\RR^n$. We assume without loss of
generality that the coordinates of the critical point $z$ are $0$
and in these coordinates the metric reads $g_{jk}(x)=\de_{jk}+O(|x|)$. We shall denote by
$P\in\RR[x_1,\dots,x_n]$ the first nonzero homogenous term of the
Taylor expansion of $G-c$ in these coordinates at the critical
point. By the harmonicity of $G$ in $(M\minus\{y\},g)$ and
a theorem of Bers~\cite{Be55}, $P$ satisfies the equation $\De P=0$, i.e., it is harmonic with respect to the Euclidean metric.

Consider the tangent cone $P^{-1}(0)$ of $G$ at $z$~\cite{Ar85}, which is
the union of the limits of secant lines passing through $z$. Proposition~\ref{P:elem} and Remark \ref{RP:elem} show that $P^{-1}(0)$ has pure codimension $1$, and thus
disconnects $\RR^n$. By hypothesis, $0$ is an isolated critical point of $P$, which implies that
\[
\big|\nabla P(x)\big|\geq C|x|^{d-1}\,,
\]
where $d$ is the degree of $P$ and
\[
C:=\inf_{|x|=1}\big|\nabla P(x)\big|>0\,.
\]
Therefore, a well known theorem of Kuiper~\cite{Ku72} shows that $G$ is $C^1$-equivalent to $P$ in $U$, so that
$U\minus\Lambda$ has as many components as $\RR^n\minus P^{-1}(0)$ if $U$ is small enough. Thus the number of connected components of
$U\minus P^{-1}(0)$ is exactly two and $P^{-1}(0)$ is necessarily
homeomorphic to a plane.

We can now prove that $P^{-1}(0)$ is indeed diffeomorphic to a
plane, which implies that $\La$ has a tangent plane at $z$. In order to see this, note that, $P^{-1}(0)$
being homeomorphic to a plane, its intersection $S$ with the unit
sphere $\{x_1^2+\cdots+x_n^2=1\}$ is homeomorphic to $\SS^{n-2}$. Besides, the homogeneity of $P$ implies that
\[
P^{-1}(0)=\big\{\la x:\la\in\RR,\; x\in S\big\}\,.
\]
Hence $P^{-1}(0)$ disconnects $U$ into more than two components unless
\[
\big\{\la x:\la\geq0,\; x\in S\big\}=\big\{\la x:\la\leq0,\; x\in
S\big\}\,,
\]
that is, unless $P^{-1}(0)$ is an affine plane which contains the
origin.

We have thus shown that $\La$ has a tangent plane at $z$, which can be taken as $\{x_1=0\}$ without loss of generality. The division
property of the polynomials then allows us to write $P=QR$, where
$R(x_1,\dots,x_n)=x_1$ and $Q$ is a homogeneous polynomial of
degree $\deg(P)-1$. Lemma~\ref{L:HarmPoly} now implies that $Q$ is
constant, contradicting the fact that $G$ has a critical point at
$z$ and completing the proof of the proposition.
\end{proof}

\begin{remark}
Being a local result, Proposition~\ref{T:regions} holds as well for
harmonic functions and for Green's functions which are not Li--Tam. The technical hypothesis that $z$ is an isolated critical point of $P$ is essential and cannot be dropped. This can be seen, e.g., by considering the harmonic polynomial in Euclidean $3$-space defined by $f(x_1,x_2,x_3):=x_1^2-x_2^2+(x_1^2+x_2^2)x_3-\frac23x_3^3$. Indeed, let $U$ be a small neighborhood of the origin, which is an isolated critical point of $f$. One can readily check that $f^{-1}(0)\cap U$ is homeomorphic to a plane, whereas the tangent cone of $f$ at $0$ is $\{x_1^2-x_2^2=0\}$.
\end{remark}

The local flow of the vector field $\frac{\nabla_g G}{|\nabla_g
G|^2_g}\in\mathfrak X^\om(M\minus(\Cr(G)\cup\{y\}))$ is a basic tool
in the global study of the level sets of $G$. As a matter of fact, if $G$ does
not have any critical points in $U:=G^{-1}((c_1,c_2))$, $U$ is a trivial bundle over $(c_1,c_2)$ and the local
flow $\phi_t$ of the latter vector field provides a $C^\om$
diffeomorphism
\[
\phi_tG^{-1}(c)=G^{-1}(c+t)
\]
whenever $c_1<c,c+t<c_2$. This is an easy consequence of the
fact that all the connected components of $G^{-1}(c)$ are compact,
cf.\ Proposition~\ref{P:elem}. In this sense, one should note the
following obvious result, which makes use of the well known asymptotics for the Green's function at the pole~\cite{GS55}:
\begin{align}
G(x)&=
\begin{cases}
C_n\,\dist_g(x,y)^{2-n}\,\big(1+o(1)\big)\qquad&\text{if }\;n\geq3\,,\\[2mm]
-C_2\,\log\dist_g(x,y)\,\big(1+o(1)\big) &\text{if }\;n=2\,,
\end{cases}\label{asympG}\\[2mm]
\big|\nabla_g G(x)\big|_g&=C'_n\,\dist_g(x,y)^{1-n}\,\big(1+o(1)\big)\,.\label{asympDG}
\end{align}
Here $C_n,C_n'$ are constants and $\lim_{x\to y}o(1)=0$.

\begin{proposition}\label{P:gil}
If $\Cr(G)$ is empty, $M$ is diffeomorphic to $\RR^n$.
\end{proposition}
\begin{proof}
Eq.~\eqref{asympG} shows that the level sets $G^{-1}(c)$ are
diffeomorphic to $\SS^{n-1}$ for sufficiently large $c$. As there
are no critical points in $M$, all the level sets of $G$ must be
diffeomorphic to $\SS^{n-1}$, and thus $M$ is the monotone union of
the balls
\[
B_j:=G^{-1}((c_j,\infty))\cup\{y\}\,.
\]
Here $c_j$ is a decreasing sequence in $(\inf G,\infty)$ converging to $\inf G$. A theorem of Brown now ensures~\cite{Br61} that
$M$ is then diffeomorphic to $\RR^n$.
\end{proof}

\begin{remark}
For future reference, let us mention that Proposition~\ref{P:gil} obviously holds true as well when $(M,g)$ is of class $C^\infty$.
\end{remark}

Further information on the level sets of $G$ can be extracted from the trajectories of $\nabla_g G$. Since this
vector field is not complete, and in fact is not even defined at the pole
$y$, it is convenient to introduce a global $C^1$ vector field
$X\in\mathfrak X^1(M)$ which has the same (unparametrized) integral
curves as $\nabla_g G$. We start by taking a function
$\varphi\in C^1(M)$ such that $\varphi(x)=1$ if $\dist_g(x,y)>\ep$
and
\[
\lim_{x\to y}\frac{\varphi(x)}{\dist_g(x,y)^n}=1\,.
\]
We assume that $\varphi$ does not vanish in $M\minus\{y\}$ and $\ep$ is
a small enough positive number. The asymptotics~\eqref{asympG}
and~\eqref{asympDG} show that the complete vector field
\begin{equation}\label{X}
X:=\frac{\varphi\nabla_g G}{\sqrt{1+\varphi^2|\nabla_g G|_g^2}}
\end{equation}
is of class $C^1$ in $M$ and analytic in
$M\minus\overline{B_g(y,\ep)}$. By construction, $X$ does not
vanish but at $\Cr(G)\cup\{y\}$, and $y$ is a local attractor.

We will denote by $\al(x)$ and $\om(x)$ the $\al$- and
$\om$-limit sets of a point $x$ along the flow of $X$. Given a point
$z\in\Cr(G)$, we define its stable and unstable sets as
\begin{align*}
W^s(z)&:=\big\{x\in M:\om(x)=z\big\}\,,\\
W^u(z)&:=\big\{x\in M:\al(x)=z\big\}\,.
\end{align*}
Let us also introduce the notation
\begin{equation}\label{eqD}
D:=\big\{x\in M: \om(x)=y\big\}
\end{equation}
for the the {\em basin of attraction} of $y$. It is
easy to check that $D$ is diffeomorphic to $\RR^n$~\cite{EP07}, while Proposition~\ref{P:gil} shows
that $M\minus D$ is nonempty whenever $M$ has nontrivial topology. These objects, whose relationship with the critical set of $G$ is laid bare in the following theorem, will play a fundamental role in the proof of Theorem~\ref{T:dim2}.

\begin{theorem}\label{T:boundary}
$\pd D=M\minus D=\bigcup\limits_{z\in\Cr(G)} W^s(z)$.
\end{theorem}
\begin{proof}
Clearly $D$ is an open set. As the function $G$ is decreasing by
Property~(i) of Li--Tam Green's functions, the restriction of $G$ to
the closed set $M\minus D$ must attain its maximum at some point
$x$. The invariance of $D$ under the flow of $X$ implies that of
$M\minus D$, which in turn ensures that $x$ is a local attractor
of the flow of $X$ in $M\minus D$. If the interior of $M\minus D$
were nonempty, this contradicts the harmonicity of $G$, i.e., the
fact that the local flow of $\nabla_g G$ preserves volume. Hence it follows that $M\minus
D=\pd D$.

As the function $G$ is analytic in $M\minus\{y\}$, the $\al$- and
$\om$-limit sets of a point $x\in M$ are either empty or a unique
point in $\Cr(G)\cup \{y\}$~\cite{Mo97}. Clearly
$W:=\bigcup_{z\in\Cr(G)} W^s(z)$ is contained in $M\minus D$: as $y$
is a local attractor, it cannot be an $\al$-limit, so necessarily
$\al(x)\in M\minus D$ for all $x\in W$.

To prove the converse implication, we shall show that the
$\om$-limit of any $x\in M\minus D$ necessarily belongs to
$\Cr(G)$. Obviously it suffices to consider the case
$x\not\in\Cr(G)$. Denoting by $\phi_t$ the flow of $X$, we have
seen that $\phi_tx$ must either tend to a critical point of $G$ or
go to infinity as $t\to\infty$. To preclude the latter possibility, let us compute
the derivative of $G$ along the flow of $X$:
\[
\cL_XG=g(X,\nabla_gG)=\frac{\varphi|\nabla_g
G|_g^2}{\sqrt{1+\varphi^2|\nabla_g G|_g^2}}>0\qquad\text{in
}M\minus(\Cr(G)\cup \{y\})\,.
\]
As $\frac{\pd}{\pd t}G(\phi_tx)=(\cL_XG)(\phi_tx)$ is strictly
positive and $G$ is decreasing in the sense of Properties~(i)
and~(ii), this ensures that the integral curve of $X$ passing
through $x$ does not go to infinity, so that $x$ necessarily lies in $W$.
\end{proof}

\begin{remark}
  Theorem~\ref{T:boundary} and its proof are also valid for $n=2$.
\end{remark}

\begin{example}
  Let us now discuss an example where the basin of attraction and the
  stable manifolds of the critical points of the Green's function can be
  computed explicitly. Consider the cylinder $M=\RR\times (\RR/2\pi\ZZ)$
  endowed with its standard flat metric. A (nonpositive) Green's
  function is given by~\cite{EP07}
\[
G(x,\theta)=-\frac1{4\pi}\log(\cosh x-\cos \theta)\,,
\]
where $(x,\theta)\in\RR\times(\RR/2\pi\ZZ)$ and the pole is located at $(0,0)$. It is not difficult to see that this Green's
function is of Li--Tam type. The basin of attraction of the gradient
field
\[
\nabla_gG=-\frac{\sinh x\,\pd_x+\sin\theta\,\pd_\theta}{4\pi(\cosh
x-\cos\theta)}
\]
can be readily shown to be the complement of the invariant line
$L:=\{\theta=\pi\}$. Clearly $G$ has a unique critical point
$z_0:=(0,\pi)$, and its stable set is precisely $L$. The level set
$G^{-1}(c)$ is diffeomorphic to a circle for any $c$ greater than the
critical value $c_0:=-\frac{1}{4\pi}\log 2$, whereas for $c<c_0$ the level
set $G^{-1}(c)$ is composed of two disjoint closed curves.
\end{example}

It stems from Proposition~\ref{P:gil} that $\Cr(G)$ is nonempty
whenever $M$ has nontrivial topology. When $M$ is diffeomorphic to
$\RR^n$, Theorem~\ref{T:boundary} imposes a topological
restriction on the possible critical sets of $G$, namely that
$M\minus\bigcup_{z\in\Cr(G)} W^s(z)$ must still be diffeomorphic
to $\RR^n$. When $M$ is diffeomorphic to a plane, one can prove that this condition is incompatible with the harmonicity of $G$ in $M\minus\{y\}$, which implies that $\Cr(G)=\emptyset$ (cf.\ e.g.~\cite{Wa50}).

In the rest of this section we shall further analyze the critical set $\Cr(G)$ 
under the assumption that $M$ is diffeomorphic to $\RR^n$. Notice, first of
all, that the contractibility of $M$ and
$M\minus\pd D$ together with the harmonicity of $G$ impose strong
restrictions on the local structure of $\pd D$. Some of these
restrictions can be qualitatively understood using only that $G$ is
analytic~\cite{Mo97,KP02}; however, the restrictions arising from the fact that
$G$ is harmonic as well play a crucial role in the analysis and are usually much harder to
characterize~\cite{Go05}.

In order to illustrate this fact, we consider next the simplest
choice for $\pd D$ which would be a priori compatible with the obvious topological
obstructions, namely that $\pd D$ includes an isolated critical
point whose stable set is a half-line. Certainly $D=M\minus\pd D$
would be diffeomorphic to $\RR^n$ for this cusp-like critical point, as required. Nonetheless, Proposition~\ref{T:regions} can be used to generically rule out this possibility:

\begin{proposition}\label{P:halfline}
Let $z$ be an isolated critical point of $G$ which is also an isolated critical point of the first nonzero homogeneous term of the Taylor expansion of $G$ at $z$. Then $W^s(z)$ and $W^u(z)$ are nonempty and not
homeomorphic to the half-line $[0,+\infty)$.
\end{proposition}
\begin{proof}
Let $c:=G(z)$ be the critical value. By
Proposition~\ref{T:regions}, $G^{-1}(c)$ disconnects a small
neighborhood $U$ of $z$ into $N\geq 3$ regions, which we shall call $R_j$ ($j=1,\dots, N$). For each
$j$, the gradient field $\nabla_g G$ points either inwards or
outwards on $\pd R_j\minus\{z\}$, that is, if $\nu_j$ denotes the
outward normal field of $\pd R_j$ then $g(\nabla_g G,\nu_j)$ does not change sign in $\pd R_j\minus\{z\}$, so it is everywhere positive or negative. Since $G(\pd R_j)=c$ for all $j$, the pointing
directions of $\nabla_g G$ in neighboring regions must alternate,
that is,
\[
g\big(\nabla_g G(x_j),\nu_j(x_j)\big)\, g\big(\nabla_g
G(x_k),\nu_k(x_k)\big)<0
\]
for all $x_j\in\pd R_j,\;x_k\in\pd R_k$ whenever $(\pd R_j\cap\pd
R_k)\minus\{z\}\neq\emptyset$ and $j\neq k$.

Therefore, if $N\geq 4$ there are at least two regions, say $R_1$
and $R_2$, in which $g(\nabla_g G,\nu_j)|_{\pd
R_j\minus\{z\}}>0$ ($j=1,2$). By Wazewski's theorem~\cite[Theorem 3.1]{Ha82},
in each region $R_j$ ($j=1,2$) there exists an integral curve of
$\nabla_g G$ whose $\om$-limit is $z$, and thus $W^s(z)$ cannot be
homeomorphic to a half-line.

From now on we shall assume that $N=3$. In this case the number of regions with $g(\nabla_g G,\nu_j)>0$ on
their boundaries is either one or two. If it is two, the previous
argument applies here as well, so we can restrict ourselves to the
case where $g(\nabla_g G,\nu_1)|_{\pd R_1\minus\{z\}}>0$ and
$g(\nabla_g G,\nu_j)|_{\pd R_j\minus\{z\}}<0$ for $j=2,3$.

The flow of $\nabla_g G$ defines a map $\Phi: G^{-1}(c-\ep)\cap
U\to G^{-1}(c)\cap U$ for sufficiently small $\ep>0$, and by the analyticity of $G$ it is
standard that its restriction
\[
\overline\Phi: \big(G^{-1}(c-\ep)\cap U\big)\minus W^s(z)\to
\big(G^{-1}(c)\cap U\big)\minus\{z\}
\]
is a diffeomorphism. As $G>c$ in $R_2\cup R_3$, the submanifold
$G^{-1}(c-\ep)\cap U$ is a codimension 1 connected subset of
$R_1$. If $W^s(z)$ is homeomorphic to a half-line, the
intersection $W^s(z)\cap G^{-1}(c-\ep)$ consists of one point, and
therefore $(G^{-1}(c-\ep)\cap U)\minus W^s(z)$ is connected.

As $z$ is an isolated critical point and $G^{-1}(c)$ has pure codimension 1, Lojasiewicz's structure theorem~\cite[Section 6.3]{KP02} guarantees that $G^{-1}(c)\cap U$ is given by the disjoint union of $\{z\}$ and $m$ open, analytic submanifolds of codimension 1. Therefore, if $m\geq2$ it trivially follows that $(G^{-1}(c)\cap U)\minus\{z\}$ is disconnected, contradicting the fact that $\overline\Phi$ is a diffeomorphism.. If $m=1$, $G^{-1}(c)\cap U$ is homeomorphic to $\RR^{n-1}$, contradicting the fact that $N=3$ and proving the claim for $W^s(z)$. All the arguments being local, the proof
obviously applies to $W^u(z)$ as well once we
change $G$ by $-G$.
\end{proof}
\begin{remark}\label{R:halfline}
Proposition~\ref{P:halfline} asserts that, generically, a connected component of
$\pd D$ cannot be homeomorphic to a half-line. The next simplest,
a priori admissible possibility would be to have a component
homeomorphic to a half-plane, as happens when one has a saddle
connection between two nondegenerate critical points of Morse
indices $1$ and $2$. 
\end{remark}

\section{Finitely generated surfaces and axisymmetric manifolds}
\label{S:surfaces}

Unless otherwise stated, in this section we shall assume that $G$ is a Li--Tam Green's function on a Riemannian surface $(M,g)$ with finitely generated fundamental group. Our objective is to prove Theorem~\ref{T:dim2}, which will hinge on a thorough analysis of the set $\pd D$  introduced in~\eqref{eqD}. The key point is to show that the critical set of $G$ is finite; once this is known, one can get the desired topological upper bound using local arguments and Hopf index theory. 

Let us begin by introducing some notation. Let $z$ be a critical point of $G$ and consider a small neighborhood $U\ni z$. We can assume that $U$ is covered by a single normal chart~\cite{SY95} $\psi:U\to\RR^2$ with $\psi(z)=(0,0)$. We also use the notation  $x=(x_1,x_2)$ for these normal coordinates.
The unique continuation theorem~\cite{Ar57} implies that the function $G-G(z)$ can vanish only up to finite order at $z$. By a theorem of Bers~\cite{Be55}, there exists a homogeneous polynomial $h_m$ of degree $m\geq2$ such that  
\begin{subequations}\label{new}
\begin{align}
G\circ\psi^{-1}(x)-G(z)=h_m(x)+O(|x|^{m+1})\,,\\
\psi_*\nabla_gG(x)=\nabla h_m(x)+O(|x|^m)\,,\label{eq2}
\end{align}
\end{subequations}
where, as usual, $\nabla$ stands for the Euclidean gradient operator in the coordinates $(x_1,x_2)$. Furthermore, one has that $\De h_m=0$, which implies that the origin is an isolated critical point of $h_m$. It follows from this discussion that all the critical points of $G$ are necessarily isolated.   

A simple but crucial property is presented in the following lemma,
which furnishes a local description of the integral curves of
$\nabla_g G$ in a neighborhood of a critical point. Here and in what follows, the Taylor expansion of $G$ at $z$ is performed in the coordinates $(x_1,x_2)$. This result is analogous to Cheng's local analysis of the nodal set of the eigenfunctions on a surface~\cite{Ch76}, our proof using a blow-up argument instead of the Kuiper--Kuo theorem.

\begin{lemma}\label{L:halfb}
Let $z$ be a critical point of $G$, and let $m\geq2$ be
the degree of the lowest nonzero homogeneous term in the Taylor
expansion of $G$ at $z$. Then there exists a neighborhood $U$ of
$z$ such that $W^s(z)\cap U$ and $W^u(z)\cap U$ are homeomorphic
to $\{\zeta\in\CC:\zeta^m\in[0,1)\}$.
\end{lemma}
\begin{proof}
In the proof of this lemma we shall use the notation introduced in Eqs.~\eqref{new} and the map
$\phi:\RR^+\times\SS^1\to\RR^2$ defined by
$\phi(r,\theta):=(r\,\cos\theta,r\,\sin\theta)$. As $h_m$ is
harmonic, it readily follows that
$h_m\circ\phi(r,\theta)=C\,r^m\cos (m\theta-\theta_0)$ for some constants $C$ and $\theta_0$. There is obviously no loss of generality in setting $\theta_0=0$.

We now define the polar blow up~\cite{SS80} of $\psi_*\nabla_gG$
at the singularity as
\[
X(r,\theta):=\frac1{Cm\,r^{m-2}}\,\phi^*\psi_*\nabla_gG(r,\theta)\,,
\]
so that, taking into account Eq.~\eqref{eq2}, the blown-up trajectories are precisely given by
\begin{subequations}\label{dotrte}
\begin{align}
\dot r&=r\,\cos m\theta+O(r^2)\,,\\
\dot\theta&=-\sin m\theta+O(r)\,.
\end{align}
\end{subequations}
The blown-up critical points are thus $(0,\theta_k)$,
with $\theta_k:=k\pi/m$ and $k=1,\dots, 2m$. The Jacobian matrix
of $X$ at $(0,\theta_k)$ is
\begin{equation}\label{Jac}
DX(0,\theta_k)=\left(%
\begin{array}{cc}
   (-1)^k& 0 \\
  0 & (-1)^{k+1} \\
\end{array}%
\right)\,,
\end{equation}
so these critical points are hyperbolic saddles. By blowing down,
we immediately find that a deleted neighborhood of $z$ in $U$
consists of $2m$ hyperbolic sectors of the original vector field
$\nabla_gG$, and that each pair of consecutive separatrices form an
angle of $\pi/m$. By~\eqref{Jac} both
$W^s(z)\cap U$ and $W^u(z)\cap U$ are
homeomorphic to $\{\zeta\in\CC:\zeta^m\in[0,1)\}$, as claimed.
\end{proof}

We shall next provide the demonstration of Theorem~\ref{T:dim2}. The proof of several technical points are relegated to
Lemmas~\ref{L:hatG}--\ref{L:homo} below, which without further mention we state in terms of
objects defined in the proof of the theorem. We recall that the $p$-th Betti number of a manifold $N$, denoted as $b_p(M)$, is
the rank of its $p$-th homology group $H_p(N;\ZZ)$, and that a critical point
$z$ of a $C^r$ function $f$ ($2\leq r\leq\om$) is called {\em
  nondegenerate} if the Hessian matrix of $f$ at $z$ has maximal
rank. A $C^r$ function is called {\em Morse}
when all its critical points are nondegenerate.

\begin{proof}[Proof of Theorem~\ref{T:dim2}]
By the classification of noncompact surfaces~\cite{Ri63}, there exists a homeomorphism
\[
\Phi: M\to\Si\minus\{h_1,\dots,h_\la\}\,,
\]
where $\Si$ is a compact oriented surface of genus $\nu$ and $\la$ points
$h_1,\dots,h_\la\in\Si$ have been deleted. It can be readily checked that the first Betti number can be expressed in terms of $\nu$ and $\la$ as
$b_1(M)=2\nu+\la-1$. We shall assume hereafter that $\Si$ is endowed with a smooth metric $\overline g$.

Let $\phi_t$ be the flow of the regularized vector field $X$ defined by Eq.~\eqref{X}, which is smooth in $M\backslash \overline{B_g(y,\epsilon)}$. We
define a continuous flow $\theta_t$ on $\Si$ by
\[
\theta_tx:=\begin{cases}\Phi\circ\phi_t\circ\Phi^{-1}(x)\,,\qquad
&\text{if
}\;x\not\in\{ h_1,\dots,h_\la\}\,,\\
x&\text{if }\;x\in\{ h_1,\dots,h_\la\}\,,
\end{cases}
\]
and denote by $\Fix(\theta_1)$ the fixed point set of the time-one map
$\theta_1$. As $X$ is proportional to the gradient field $\nabla_gG$, it does not have any periodic
trajectories, so that
\[
\Fix(\theta_1)=\Phi\big(\Cr(G)\cup\{y\}\big)\cup\{h_1,\dots,h_\la\}\,.
\]
The fact that the critical points of $G$ are isolated guarantees that $\Fix(\theta_1)$ does not accumulate but
possibly at $\{h_1,\dots,h_\la\}$.

Let $\gamma$ be a $\theta$-trajectory, i.e., a map
$\gamma:\RR\to\Si$ given by $\gamma(t)=\theta_tx$ for some
$x\in\Si$; this trajectory will be called {constant} if
$\gamma(\RR)=\{x\}$. Let us introduce the notation $\al(\gamma)$
and $\om(\gamma)$ for the $\al$- and $\om$-limit sets of the
trajectory $\gamma$. The $\al$- and $\om$-limit sets $\al(S)$, $\om(S)$ of a
subset $S\subset\Si$ are defined in the standard fashion, that is, as the union of the
$\al$- or $\om$-limits of the $\theta$-orbits passing through $S$. Since
$\nabla_gG$ is the gradient of a smooth function in
$M\minus\{y\}$ whose critical points are isolated, $\Phi$ is a homeomorphism and $\Si$ is a
$\la$-point compactification of $M$, it follows by standard
arguments~\cite[Section 1.1]{MP82} that
$\al(\gamma)$ and $\om(\gamma)$ consist of a single
point, which lies in $\Fix(\theta_1)$ and is respectively given by
\[
\al(\gamma)=\lim_{t\to-\infty}\gamma(t)\,,\qquad
\om(\gamma)=\lim_{t\to\infty}\gamma(t)\,.
\]

Let us introduce a partial order on $\Fix(\theta_1)$ as follows. Given two
points $x,x'\in \Fix(\theta_1)$, we shall write $x\prec x'$ if, for any
open neighborhoods $U\ni x$ and $V\ni x'$ in $\Sigma$, there exist integers
$p\leq0,q\geq1$ and nonconstant $\theta$-trajectories
$\gamma_{p},\dots,\gamma_{q}$ such that
\begin{enumerate}
\item $\om(\gamma_{p})\in U$, $\al(\gamma_{q})\in V$,
\item $\al(\gamma_j)=\om(\gamma_{j+1})$ for $p\leq j\leq q-1$.
\end{enumerate}

We claim that $G$ has a finite number of critical points. In order to
prove this, let us assume the contrary. By Lemma~\ref{L:h}, for each point
$x\in\Phi(\Cr(G))$ there exists some $j\in\{1,\dots,\la\}$ such that
$x\prec h_j$. Therefore there exists some $j$, say $j=1$, such that one
can choose a sequence $(x_k)_{k=1}^\infty$ of distinct points in
$\Phi(\Cr(G))$ with $x_k\prec h_1$. For each point $x_k$,
Lemma~\ref{L:curve} yields a continuous path $\Gamma_{k,1}:[0,1]\to\Si$ whose image is
invariant under $\theta_t$ and satisfies $\Gamma_{k,1}(0)=x_k$ and
$\Gamma_{k,1}(1)=h_1$.

By Lemma~\ref{L:halfb}, the $\theta_t$-invariant set
$\Phi(W^s(\Phi^{-1}(x_k)))\minus\Gamma_{k,1}([0,1])$ is nonemp\-ty. If we
let $\tilde x_k$ be the $\al$-limit of any $\theta$-orbit contained in
this set, we obviously have $x_k\prec\tilde x_k$. Since it follows from Lemma~\ref{L:h} that either $\tilde
x_k\in\{h_1,\dots, h_\la\}$ or $\tilde x_k\prec h_j$ for some $j$, by Lemma~\ref{L:curve} and possibly upon restricting
ourselves to a subsequence of $(x_k)$ we obtain a family of continuous
paths $\Gamma_{k,2}:[0,1]\to\Si$ whose image is invariant under $\theta_t$ and such
that $\Gamma_{k,2}(0)=x_k$ and $\Gamma_{k,2}(1)=h_j$ for some fixed $j$
(possibly $1$).

By construction, for each $k\in\NN$ the connected set
\[
\bigcup_{l,m=0,1}\Gamma_{k+l,1+m}([0,1])
\]
contains a subset $\La_k$ which is homeomorphic to $\SS^1$. One can
obviously ensure that $\La_j\neq\La_k$ for $j\neq k$. By
Lemma~\ref{L:homo}, all these loops must define different homology classes $[\La_j]\in H_1(\Si;\ZZ)$,
which contradicts the fact that the fundamental group of $\Si$ is finitely
generated. It then follows that $\Cr(G)$ is finite.

Let $B_j\subset\Si$ be a small neighborhood of $h_j$, which we can assume
to be disjoint from $\Phi(\Cr(G))$ by the finiteness of $\Cr(G)$. There is no loss of generality in assuming that the restriction of $\Phi^{-1}$ to $\Si\minus\bigcup_{j=1}^\la B_j$ is a diffeomorphism onto its image. As the number of critical points of $G$ is
finite, we can use Hopf's index theorem for the manifold with boundary
$N:=\Phi^{-1}(\Si\minus\bigcup_{j=1}^\la B_i)\subset M$. Note that the
vector field $X$ is transverse to $\pd N$ because $h_j$ is an isolated
minimum of the function $\hat G$ defined in Lemma~\ref{L:hatG}. Let $\{z_i\}_{i=1}^N$ be the critical points of $G$,
and let $m_i$ be the corresponding integer introduced in
Lemma~\ref{L:halfb}. Since
\[
\ind(X;z_i)=\ind(\nabla_gG;z_i)=1-m_i\,,
\]
and $\ind(X,y)=1$ by~\eqref{dotrte} and~\cite{Ar85}, a straightforward application of Hopf's index
theorem \cite{Ar85} shows that
\[
\chi(N)=1+\sum_{i=1}^N\ind(X,z_i)=1+N-\sum_{i=1}^Nm_i\,,
\]
$\chi(N)=2-2\nu-\la$ being the Euler characteristic of $N$. As
$m_i\geq 2$, it readily follows that
\[
N\leq 2\nu+\la-1=b_1(M)\,.
\]
Furthermore, the equality holds when $m_i=2$ for all $i$, which implies that all the
critical points of $G$ are nondegenerate.
\end{proof}

\begin{remark}\label{R:div}
For later reference, note that the proof remains valid if we only assume that $G$ satisfies the elliptic PDE in divergence form
$\Div_\om(\nabla_g G)=-\de_y$ and fulfills the conditions (i) and (ii) of Li--Tam Green's functions. Here $\om$ stands for an arbitrary
volume form in $M$ of class $C^\infty$.
\end{remark}

\begin{lemma}\label{L:hatG}
There exists a unique continuous function $\hat
G:\Si\to[-\infty,+\infty]$ such that $\hat
G|_{\Si\minus\{h_1,\dots, h_\la\}}=G\circ\Phi^{-1}$. Moreover, the
points $h_j$ are isolated local minima of $\hat G$.
\end{lemma}
\begin{proof}
The fact that $G$ is Li--Tam implies the existence of the
limit
\[
\lim_{x\to h_i}G\circ\Phi^{-1}(x)
\]
(possibly $-\infty$) for all $1\leq i\leq\la$. As $\lim_{x\to y}G(x)=+\infty$, this
enables us to extend $G\circ\Phi^{-1}$ naturally to a continuous
function $\hat G:\Si\to[-\infty,+\infty]$. By Property~(i) of Li--Tam Green's functions, the
points $h_j$ are isolated local minima of $\hat G$.
\end{proof}

\begin{lemma}\label{L:curve}
Let $x,x'\in\Fix(\theta_1)$ such that $x\prec x'$. Then there
exists a (not necessarily unique) continuous path
$\Gamma:[0,1]\to\Si$ such that:
\begin{enumerate}
\item $\Gamma(0)=x$ and $\Gamma(1)=x'$.

\item $\hat G\circ\Gamma$ is strictly decreasing, where $\hat G$
is the function defined in Lemma~\ref{L:hatG}.

\item The curve $\Gamma([0,1])$ is invariant under the flow
$\theta_t$.
\end{enumerate}
\end{lemma}
\begin{proof}
By the definition of the partial order $\prec$ and Zorn's lemma, there exists a countable
sequence $\{\gamma_j\}_{j=\overline p}^{\overline q}$ ($-\overline p,\overline q\in\NN\cup\{\infty\}$) of
nonconstant $\theta$-trajectories satisfying Conditions (i) and (ii) in the proof of Theorem~\ref{T:dim2}, and
such that
\[
\lim_{j\to\overline  p}\om(\ga_j)=x\,,\qquad \lim_{j\to \overline q}\al(\ga_j)=x'\,.
\]
Hence any homeomorphism
\[
\Gamma:[0,1]\to\overline{\bigcup_{j=\overline p}^{\overline q}\gamma_j(\RR)}\subset\Si
\]
mapping $0$ to $x$ and $1$ to $x'$ yields the desired path. Since
the Lie derivative of $G$ along $X$ is nonnegative and strictly
positive in $M\minus\Cr(G)$, $\hat G$ is easily shown to be
increasing along $\theta$-orbits, which implies that $\hat G\circ\Gamma$ is
strictly decreasing. Moreover, the set $\Gamma([0,1])$ is clearly invariant
under $\theta_t$ because it is the union of $\theta$-orbits.
\end{proof}

\begin{lemma}\label{L:h}
For any $x\in\Fix(\theta_1)\minus\{h_1,\dots,h_\la\}$ there exists
some $j$ such that $x\prec h_j$. Moreover, $h_1,\dots,h_\la$ are
maximal elements of $\prec$ and $x\not\prec x$ for any
$x\in\Fix(\theta_1)$.
\end{lemma}
\begin{proof}
Let us take an element $x\in\Fix(\theta_1)\minus\{\Phi(y),h_1,\dots,h_\la\}$. By
Lemma~\ref{L:halfb}, there exists a neighborhood $U\subset M$ of
$z:=\Phi^{-1}(x)$ such that $(W^s(z)\cap U)\minus\{z\}$ has at
least two components $C_1,C_2$ and each $C_i$ is a connected subset of
an orbit of $X$. As $G$ is smooth and its critical set consists of isolated points, either the $\al$-limit set
of $C_1$ under the flow of $X$ does not exist or it is another
critical point $z_1$ of $G$.

Set $x_1:=\Phi(z_1)$. If the $\al$-limit of $C_1$ does not exist in $M$,
there exists a sequence $(t_k)_{k=1}^\infty\searrow-\infty$ such that
$\dist_{\overline g}(\theta_{t_k}x_1,\{h_1,\dots,h_\la\})\to0$. In turn, this implies that
\[
\liminf_{k\to\infty}
\dist_{\overline g}(\theta_{t_k}x_1,h_j)=0
\]
for some $j\in\{1,\dots,\la\}$ because $\la$ is finite. The facts that
$h_j$ is an isolated local minimum of $\hat G$ by
Lemma~\ref{L:hatG} and $\hat G$ is increasing along
$\theta$-orbits now imply that $\theta_{t}x_1$ tends to $h_j$ as
$t\to-\infty$, proving the statement.

Let us now assume that the $\al$-limit of $C_1$ is another
critical point $z_2\in\Cr(G)$ and repeat the previous argument
replacing $z_1$ by $z_2$. Proceeding this way we obtain a sequence
of critical points $z_1,z_2,\dots$ of $G$. If this sequence is finite
and terminates after $k$ steps, the stable set $W^s(z_k)$ of the critical point
$z_k$ must be unbounded, and in this case the
previous argument readily shows that $\theta_{t}\Phi(z_k)$ tends
to some point $h_j$ as $t\to-\infty$.

If the latter sequence is infinite, we find a countable sequence
$(z_k)_{k=1}^\infty\subset\Cr(G)$ with $x_k\prec x_{k+1}$, where
$x_k:=\Phi(z_k)$. As $\Cr(G)$ consists of isolated points in $M$, it follows that $\dist_{\overline g}(x_k,\{h_1,\dots,h_\la\})$ tends to zero as
$k\to\infty$, and hence there must exist some $j\in\{1,\dots,\la\}$ such that
\[
\lim_{k\to\infty} \dist_{\overline g}(x_k,h_j)=0
\]
for some subsequence that we still call $x_k$. As above, it then follows that $x_k\to h_j$ by the fact that $h_j$
is an isolated minimum of $\hat G$ and $\hat G$ is increasing
along $\theta$-orbits.

Since the $\al$-limit of an integral
curve whose $\om$-limit is $y$ either does not exist or is a
critical point of $G$, the above proof also applies when we replace the original fixed point $x\in\Fix(\theta_1)\minus\{\Phi(y),h_1,\dots,h_\la\}$ by the pole $\Phi(y)$. Finally, note that obviously $x\not\prec x$ for any $x\in\Fix(\theta_1)$ by
Lemma~\ref{L:curve}, and that $h_j\not\prec x$ because each $h_j$ is an
isolated minimum of $\hat G$ and $\hat G$ is increasing along the
continuous flow $\theta_t$.
\end{proof}

\begin{lemma}\label{L:homo}
  Let $\La_k\subset\Si$, with $k=0,1$, be homeomorphic to $\SS^1$. Assume
  that the sets $\La_k$ do not contain $\Phi(y)$ and are invariant
  under $\theta_t$. Then the homology class of $\La_k$ is nontrivial,
  and $\La_0$ and $\La_1$ are not homologous if $\La_0\neq\La_1$.
\end{lemma}
\begin{proof}
Assume that $\La_0$ is homologous to $\La_1$ and not homologous to zero. Then there exists a unique
compact, connected set $K_0\subset\Si$ which does not contain $\Phi(y)$
and whose boundary is $\pd K_0=\La_0\cup\La_1$. Let us define
\[
K:=\Phi^{-1}\big(K_0\minus\{h_1,\dots,h_\la\}\big)\,,
\]
which is a possibly unbounded closed, connected subset of $M$
not containing~$y$. As $\La_k$ is invariant under $\theta_t$, $K$ must be
invariant under the flow of $X$.

Let us choose $x$ so that $\hat G(x)=\max_{K_0}\hat G$. Since each $h_j$ is an
isolated local minimum of $\hat G$, the point $z:=\Phi^{-1}(x)\in K$ must
satisfy $G(z)=\sup_KG$. By the maximum principle and the fact that $G$ is
harmonic in $K$, $z$ must lie on the boundary of $K$. Since $\pd K$ is
invariant and $\Cr(G)$ does not accumulate, $z$ must be the $\om$-limit
along the flow of $X$ of any other boundary point sufficiently close to
$z$. $\nabla_g G$ being a gradient field, this implies~\cite[Section 1.1]{MP82}
that $z\in\Cr(G)$.

Since $G$ is continuous on $K$, $z$ is an isolated maximum of
$G|_K$ and $G$ increases along the flow $\phi_t$ of $X$, it follows that
for any $\ep>0$ there exists some $\de>0$ such that
\[
\phi_t\big( B_g(z,\de)\cap K\big)\subset B_g(z,\ep)
\]
for all $t>0$. Hence there exists a region $B_g(z,\de)\cap K$ of nonzero
measure whose $\om$-limit is $z$. As the flow of $X$ and the local flow of
$\nabla_g G$ have the same integral curves, the fact that
$\om(B_g(z,\de)\cap K)=\{z\}$ contradicts the harmonicity of $G$ in $K$.

Likewise, if $\La_0$ is homologous to zero, it is the boundary of a
compact subset $K_0\subset \Si$ of nonempty interior which does not contain $\Phi(y)$, and the previous
argument immediately leads to a contradiction.
\end{proof}

When $M$ is diffeomorphic to a plane or a cylinder, it is a trivial consequence of Theorem~\ref{T:dim2} that a stronger statement holds, which can be regarded as an extension to Riemannian surfaces of the classical result that the Dirichlet Green's function of a simply or doubly connected region in the Euclidean plane has exactly zero or one critical points, in each case~\cite[Section 7.5]{Wa50}.

\begin{corollary}\label{C.plane}
  If $M$ is diffeomorphic to the plane, $\Cr(G)=\emptyset$.  If $M$ is
  diffeomorphic to the cylinder $\SS^1\times\RR$, then any Li--Tam Green's
  function is Morse and has exactly one critical point.
\end{corollary}
\begin{proof}
By Theorem~\ref{T:dim2}, $G$ does not have any critical points when $M\cong\RR^2$, and it has at most one critical point if $M\cong\SS^1\times\RR$. As $\Cr(G)$ is necessarily nonempty when $M$ is not contractible by Proposition~\ref{P:gil}, the claim follows.
\end{proof}

It is an open problem to show that the number of connected
components of the critical set of any Li--Tam Green's function $G$ is finite for every finitely
generated Riemannian $n$-manifold ($n\geq3$) or else to construct a counterexample. The number of components of $\Cr(G)$ is
certainly finite for analytic $n$-manifolds with nonnegative Ricci
curvature, Euclidean volume growth and quadratic curvature decay
by a theorem of Colding and Minicozzi~\cite{CM97}, which ensures
that there exists a constant $C>0$ such that
\[
\Big|\big|\nabla_g
G(x)\big|_g-C\dist_g(x,y)^{1-n}\Big|<\ep(\dist_g(x,y))\,\dist_g(x,y)^{1-n}\,,
\]
where $\lim\limits_{r\to\infty}\ep(r)=0$. Hence the critical set of
$G$ is bounded, so that the analyticity of $G$ implies that $\Cr(G)$ must have a finite
number of connected components.

The techniques that we have presented for the study of Green's functions on surfaces can be  modified to deal with Li--Tam Green's functions in axisymmetric smooth $n$-manifolds $(\RR^n,g)$. In order to illustrate this fact, we shall conclude this section with a criterion for the absence of critical points which makes use of some of the ideas introduced in the proof of Theorem~\ref{T:dim2} and in Section~\ref{S:critical}. Some details will be barely sketched to avoid unnecessary repetitions.

\begin{theorem}\label{P:absence}
Let $M$ be diffeomorphic to $\RR^n$, $n\geq3$, and suppose that there exists
a subgroup of isometries $H\subset\Isom(M,g)$ isomorphic to
$\SO(n-1)$. If $y$ is invariant under the action of $H$, then
$\Cr(G)=\emptyset$.
\end{theorem}
\begin{proof}
It is well known that the hypotheses on $H$ imply that there
exists a subset $L\subset M$, fixed under the action of $H$ and diffeomorphic to the real line, such that the action of $H$ is proper and free on
$M\minus L$. As a consequence of this, the orbit space
\[
\mathcal O:=\big(M\minus L\big)/H
\]
is a two-dimensional differentiable manifold diffeomorphic to
$\RR^2$ and the projector $\pi:M\minus L\to\cO$ is smooth.

By the property (iii) of Li--Tam Green's functions, $G$ is
invariant under the isometries of the manifold, so that both $G$
and its gradient $\nabla_gG$ can be pushed forward to the quotient space
$\mathcal O$. We shall use the notation $Y:=\pi_*(\nabla_g G)$ for the reduction of the
$G$-equivariant vector field $\nabla_g G$. The Riemannian volume
form of $M$ and the action of $H$ induce a volume form $\om$ in $\cO$ (the Liouville volume
form, cf.~\cite[Section 3.4]{AM80}) such that $Y$ is divergence-free with respect to $\om$. One should
observe that if $z$ is a critical point of $\nabla_gG$ in $M\minus
L$, $Y$ vanishes at the base point $\pi(z)$.

Under the assumption that $G$ does not have any critical points in
$L$, we shall next show that $Y$ does not vanish in $\cO$ either.
We shall call $h$ the induced metric on $\cO$ and denote by $\nabla_h$
its corresponding Levi-Civita connection. Since $H$ acts by
isometries on $M$, it immediately follows that $Y=\nabla_h F$, where
$F:=\pi_*(G)\in C^\infty(\cO)$ stands for the reduction of the
$H$-invariant function $G$ to the orbit space. If we denote
by $\Div_\om$ the divergence with respect to the volume form $\om$
and use that $\Div_\om(Y)=0$, it stems that
\begin{equation}\label{divom}
\Div_\om(\nabla_h F)=0\,.
\end{equation}

The completion of the incomplete manifold  $(\cO,h)$ is a
Riemannian manifold with boundary
$(\overline\cO,\pd\overline\cO,\overline h)$. It is not difficult to see that there exists a
unique point $\overline y\in\pd\overline\cO$ such that the
function $F$ can be continuously extended to
$\overline\cO\minus\{\overline y\}$. This extension $\overline F$
satisfies the properties (i) and (ii) of Li--Tam Green's functions and tends to $+\infty$ as one
approaches $\overline y$. An easy modification of the
proof of Theorem~\ref{T:dim2} together with Remark~\ref{R:div} can
now be used to prove that $\overline F$ does not have any critical
points in $\cO$.

It only remains to consider the case where $G$ has a critical point $z\in L$. Let us take a global chart $x=(x_1,\dots,x_n)$ and assume without loss of
generality that the coordinates of the critical point $z$ are $0$
and the metric is read in these coordinates as $g_{jk}(x)=\de_{jk}+O(|x|)$. The unique continuation property for elliptic equations~\cite{Ar57} implies that $G-G(z)$ can vanish only up to finite order at $z$, so that the first nonzero homogenous term $P$ of the
Taylor expansion of $G-G(z)$ in the coordinates $(x_1,\dots,x_n)$ at the critical
point $z$ is well defined. Let $d$ be the degree of $P$, and observe that $P$ is harmonic with respect to the Euclidean metric~\cite{Be55} and invariant under the induced action of $H$ by the $H$-invariance of $G$. There is no loss of generality in assuming that $H$ is generated by the vector fields $X_j:=x_1\frac\pd{\pd {x_j}}-x_j\frac\pd{\pd x_1}$, with $j=2,\dots,n-1$, which implies that $L=\{x_1=\cdots=x_{n-1}=0\}$.

We shall next show that $0$ is an isolated critical point of $P$. In order to see this, we start by noticing that $P$ cannot have any critical points in $\RR^n\minus L$: otherwise $\Cr(P)$ would have codimension 1 because of the homogeneity of $P$ and the fact that the $H$-orbits in $\RR^n\minus L$ have codimension 2. In turn, this would contradict the harmonicity of $P$ by Proposition~\ref{P:elem}.

Furthermore, if $P$ has other critical points on $L$, it is not difficult to prove that $\Cr(P)=L\subset P^{-1}(0)$ and that
\[P(x_1,\cdots,x_n)=\big(x_1^2+\cdots+x_{n-1}^2\big)^k\,Q(x_1,\cdots,x_n)\,,\]
where $k$ is a positive integer and $Q$ is an $H$-invariant homogeneous polynomial such that $Q^{-1}(0)\cap L=\{0\}$. It then follows that $P^{-1}(0)$ consists of the line $L$ and possibly a finite union of cones passing through the origin. But $P^{-1}(0)$ must have pure codimension 1 because $P$ is harmonic (cf.\ Proposition~\ref{P:elem}), yielding a contradiction and proving that $\Cr(P)=\{0\}$.

As $\Cr(P)=\{0\}$, it trivially follows that
\[
\big|\nabla P(x)\big|\geq C|x|^{d-1}
\]
with
\[
C:=\inf_{|x|=1}\big|\nabla P(x)\big|>0\,.
\]
Hence a theorem of Kuiper~\cite{Ku72} shows that $G-G(z)$ is $C^1$-equivalent to $P$ in a small neighborhood $U \ni z$, so that $z$ is an isolated critical point of $G$ and $U\backslash G^{-1}(G(z))$ has as many components as $\RR^n\backslash P^{-1}(0)$. Denote by $N$ the number of components, we recall from the proof of Proposition~\ref{T:regions} that  $N\geq 3$. By Wazewski's theorem~\cite[Theorem 3.1]{Ha82}, in each connected component $R_j$ of $U\backslash G^{-1}(G(z))$  ($j=1,\dots, N$) one can find an integral curve of
$\nabla_g G$ whose $\al$- or $\om$-limit is $z$, thus enabling us to conclude that there exists an integral curve $\gamma$ of $\nabla_gG$ which is contained in $M\minus L$ and
whose $\al$- or $\om$-limit is $z$.

The symmetry of $G$ ensures that $\pi\circ\gamma$ is an integral curve of
$\nabla_{\overline h}\overline F$, whose $\al$- or its $\om$-limit
set must belong to the boundary $\pd\overline\cO$. The opposite limit set of this orbit (i.e.,
$\om(\pi\circ\gamma)$ or $\al(\pi\circ\gamma)$ in each case) either lies on the boundary,
does not exist or is a critical point of $\overline F$. If one now considers the latter limit set and argues  as in the proof of
Theorem~\ref{T:dim2}, one readily arrives at a
contradiction, thereby completing the proof of the theorem.
\end{proof}

\section{Dirichlet Green's functions with nondegenerate critical points}
\label{S:Morse}

Throughout this section, the Riemannian manifold $(M,g)$ will be assumed real analytic. Our purpose here is to show that the Dirichlet Green's function of a
generic bounded domain $\Om\subset M$ of class $C^k$ is Morse for all
$k\geq2$, which will be instrumental in the proof of
Theorem~\ref{T:everything}. The usefulness of the nondegeneracy
condition for our purposes lies in the fact that it ensures that the
function is locally $C^{r-1}$-conjugate to its second order Taylor
expansion at the critical point~\cite{Hi76}, so that the structure of the
neighboring level sets can be easily controlled.  (One should notice,
however that the gradients of the function and of its quadratic part do not
need to be $C^1$-conjugate even when the functions are
$C^\om$-conjugate since the Hartman--Grobman theorem~\cite{AR67} only
grants topological conjugacy of the gradients.)

Before proving the main result of this section we find it convenient
to introduce the following definition. (We can define the $C^k(M)$ norm below using the covariant derivative. As an aside, notice that the precise way in which one defines the norm in $C^k(M)$ is inessential for our purposes because $\Phi-\id$ has compact support.)

\begin{definition}
  Two bounded $C^k$ domains $\Om,\Om'\subset M$ are said to be {\em
    $(\ep,k)$-close} if there exists a $C^k$ diffeomorphism $\Phi$ mapping
  $(\Om,\pd\Om)$ onto $(\Om',\pd\Om')$ and such that $\Phi-\id$ is compactly supported and satisfies $\|\Phi-\id\|_{C^k}<\ep$.
\end{definition}

\begin{theorem}\label{T:Morse}
  For any $2\leq k\leq \infty$, let $\Om\subset M$ be a $C^k$ bounded domain and fix a point $y\in\Om$. Then for any $\ep>0$ there exists a smooth domain $\Om'$ $(\ep,k)$-close to $\Om$ whose Green's function $G_{\Om'}$ with pole $y$ is Morse. Furthermore, there exists some $\de>0$ such that the Green's function with pole $y$ of any $C^k$ domain $(\de',k)$-close to $\Om'$ is also Morse for all $\de'<\de$.
\end{theorem}
\begin{remark}
  Equivalently, the theorem can be restated as follows: Given any $C^k$ bounded domain $\Om\subset M$, $G_{\Phi(\Om)}$ is Morse for a $C^k$-generic embedding $\Phi:\BOm\to M$.
\end{remark}
\begin{proof}
Let us divide the proof in two parts.

\hspace*{-\parindent}{\em Density.} We shall show that there
exists another domain $\Om'\subset M$ of class $C^\infty$ which is
$(\ep,k)$-close to $\Om$ and such that its associated Green's function
$G_{\Om'}$ with pole $y$ is Morse in $\overline{\Om'}\minus\{y\}$.

We start by noticing that $G_\Om\in C^\infty(\Om\minus\{y\})\cap C^{k-1,\be}(\BOm\minus\{y\})$, and that the gradient of $G_{\Om}$
is nonzero on $\pd\Om$ by the Hopf boundary point
lemma~\cite{GT98}. Hence we can choose a small enough constant $\eta>0$ such that the smooth domain $\Om_0\subset\Om$ enclosed by $G_\Om^{-1}(\eta)$ is
$(\frac\ep2,k)$-close to $\Om$. Notice that $\nabla_g G_{\Om}$ does not vanish in a
neighborhood of $y$ by Eq.~\eqref{asympDG}.

The critical set of $G_{\Om}$ being compact, we can choose a finite
number of harmonic charts
$\{(U_a,\vp_a=(x_1^a,\dots,x_n^a))\}_{a=1}^N$~\cite{GW75} so that
$\overline V_a\subset U_a\subset \Om$, $\{V_a\}_{a=1}^N$ is a finite open cover of $\Cr(G_{\Om})$, and the local
coordinates $x_j^a:U_a\to\RR$ satisfy $\De_gx_j^a=0$. As $M\minus U_a$ does not have any compact components and $(M,g)$ is analytic, the
Lax--Malgrange theorem~\cite[Section 3.10.7]{Na68} asserts that these local
harmonic coordinates can be approximated in the $C^\infty(U_a)$ weak
topology by global harmonic functions $X_j^a:M\to\RR$.

We shall prove below that there exist constants $\la_j^a$ arbitrarily close to zero
such that the function
\[
f_N:=G_{\Om}-\sum_{a=1}^N\sum_{j=1}^n\la^a_j X_j^a
\]
is Morse in the closure of
$\cV_N:=\bigcup_{a=1}^NV_a\subset\Om$. As $f_N$ obviously approximates
$G_{\Om}$ in the $C^k(\Om\minus\{y\})$ weak topology and $\nabla_g
G_{\Om}\neq0$ both in $\pd\Om$ and in a neighborhood of $y$, it will
immediately follow from the former claim and the boundedness of $\Om$
that $f_N$ is Morse in $\overline{\Om_0}\minus\{y\}$ for sufficiently
small values of $\la^j_a$.

We shall prove the above claim by induction. Let us begin by showing that there exist constants $\la^1_j$ in an arbitrarily small neighborhood of $0$ such that
\[
f_1:=G_{\Om}-\sum_{j=1}^n\la_j^1 X_j^1
\]
is Morse in $\overline{V_1}$.  Indeed, as $G_{\Om}$ is smooth, Sard's
theorem  ensures that there exists an open and dense subset
$\La_1\subset\RR^n$ such that the Hessian of $G_{\Om}\circ\vp_1^{-1}$
is nonsingular on
\[
\big\{x\in\vp_1(U_1):\nabla(G_{\Om}\circ\vp_1^{-1})(x)=\la\big\}
\]
for all
$\la\in\La_1$. By the stability of Morse functions, the assertion then follows by choosing a small enough $\la\in\La_1$ and setting $\la=(\la^1_1,\dots,\la^1_n)$,
since, by the definition of $X_j^1$, one necessarily has that
\[
\nabla(f_1\circ\vp_1^{-1})(x)=\nabla(G_{\Om}\circ\vp_1^{-1})(x)-J_1(x)\,\la\,,
\]
where the smooth function $J_1: \vp_1(U_1)\to{\rm Mat}(\RR^n)$ is arbitrarily close to the identity in the $C^\infty$ weak topology.

In order to complete the proof, let us next assume as induction
hypothesis that the function
\[
f_{m-1}:=G_{\Om}-\sum_{a=1}^{m-1}\sum_{j=1}^n\la^a_j X_j^a
\]
is Morse in the closure of $\cV_{m-1}:=\bigcup_{a=1}^{m-1}V_a$ for an open and dense set of
values of $(\la^a_j)$. If we now apply the same Sard-type
argument used for $f_1$ to the function $f_{m-1}$, we immediately derive that the function
\[
f_m:=f_{m-1}-\sum_{j=1}^n\la_j^m X_j^m
\]
is Morse in $V_m$ for an open
and dense set $\La_m\subset\RR^n$ of values of $(\la^m_1,\dots,\la^m_1)$. Moreover, the $C^2$-openness of
Morse functions in the compact manifold
$\overline{\cV_{m-1}}$ guarantees that $f_m$ is also
Morse in $\overline{\cV_{m-1}}$ provided that the new parameters
$\la_j^m$ are small enough. By induction in $m$, this proves that there exist arbitrarily small $\la^a_j\in\RR$ such that $f_N$ is Morse in $\overline{\cV_N}$.

By construction, $f_N$ is harmonic in
$\Om\minus\{y\}$. Besides, Thom's isotopy lemma~\cite[Section 20.2]{AR67} ensures that $f_N^{-1}(\eta)\subset\Om$ and $\pd\Om_0$ are $(\frac\ep2,k)$-close provided that the parameters $\la_j^a$ are chosen close enough to zero. Hence the first part of the theorem now follows by defining $\Om'$ to be the bounded domain enclosed by
$f_N^{-1}(\eta)$, so that $G_{\Om'}=f_N-\eta$.\\[3mm]
{\em Openness.} Suppose that $G_\Om$ is Morse. It is clear that
for any $\ep>0$ there exists $\de_1>0$ such that
\[
\max_{\pd(\Om\cap\Om')}G_\Om<\ep\,,\qquad
\max_{\pd(\Om\cap\Om')}G_{\Om'}<\ep
\]
for any $C^k$ domain $\Om'$ which is $(\de_1,k)$-close to $\Om$. In particular,
\begin{equation}\label{abs}
\big|G_\Om-G_{\Om'}\big|<2\ep
\end{equation}
in $\pd(\Om\cap\Om')$. The function $G_\Om-G_{\Om'}$ being
harmonic in $\Om\cap\Om'$, the estimate~\eqref{abs} must hold in
$\Om\cap\Om'$ as well by virtue of the maximum principle, and hence it
follows that $G_{\Om'}$ approximates $G_\Om$ in the
$C^0((\Om\cap\Om')\minus\{y\})$ weak topology as $\Om'$ becomes
close to $\Om$. By Harnack's theorem~\cite{GT98} this
yields weak $C^r$ approximation in $(\Om\cap\Om')\minus\{y\}$ for any $r\in \NN$, and
therefore a simple stability argument ensures that $G_{\Om'}$ is Morse
in any compact subset of $(\Om\cap\Om')\minus\{y\}$. As the gradient
of $G_{\Om'}$ does not vanish either in a neighborhood of~$y$ by the
asymptotics~\eqref{asympDG} or at $\pd\Om$ by Hopf's boundary point
lemma, the theorem follows.
\end{proof}

%
%

\section{Green's functions with prescribed behavior}
\label{S:main}

In this section we prove Theorem~\ref{T:everything}, which  follows from Corollaries~\ref{C:equip}-\ref{C:cp} and Theorem~\ref{T:codim3} below. It is plain that the results presented in this section are in strong contrast with the two-dimensional case, where the level sets of a Li--Tam Green's function on any contractible Riemannian surface are all diffeomorphic to circles as $G$ has no critical points by Corollary~\ref{C.plane}.

The proof, which is of interest in itself, is based on the construction of a $C^\om$ metric in $\RR^n$ whose minimal Green's function approximates the Euclidean Dirichlet Green's function of a prescribed domain $\Om$ (Theorem~\ref{T:main}). This is accomplished by choosing a conformally flat $C^\om$ metric on $\RR^n$ whose conformal factor is approximately 1 in $\Om$ and large in its complement. Suitable decay conditions on the curvature ensure the existence of a minimal Green's function, which permits to complete the proof of the theorem using variational methods.

According to Proposition~\ref{P:gil}, the critical set of a minimal Green's function is automatically nonempty on any noncontractible manifold, which explains why in this section we are mainly interested in analytic metrics on $\RR^n$. Nevertheless, most of the constructions we present below can be painlessly extended to other topologies.

\begin{theorem}\label{T:main}
Let $\Om\subset\RR^n$ be a $C^\infty$ bounded domain with connected boundary, and let $G_\Om$ be its Euclidean Green's function with pole $y$. Then there exists a sequence of
complete analytic metrics $g_j$ in $\RR^n$ whose
minimal Green's functions with pole $y$ approximate $G_\Om$ in the
$C^k(\Om\minus\{y\})$ weak topology, for any $k\in\NN$ and $y\in\Om$.
\end{theorem}
\begin{proof}
One can assume without loss of generality that $y$ is located at the origin of a Cartesian coordinate system $(x_1,\dots, x_n)$. Let $\tilde\vp_j:\RR^n\to[1,\infty)$ be
a smooth function such that $\tilde\varphi_j(x)=1$ if $x\in \Om$
and $\tilde\varphi_j(x)=j$ if $\dist(x,\Om)>\frac1j$. By Whitney's
approximation theorem~\cite[Section 1.6.5]{Na68}, for any $j\in\NN$ there exists an
analytic function $\varphi_j:\RR^n\to\RR$ such that
\begin{equation}\label{Dal}
\sum_{|\al|\leq\max\{3,k\}}\big|D^\al\varphi_j(x)-D^\al\tilde\varphi_j(x)\big|<\frac{\e^{-|x|}}j\,.
\end{equation}
There is no loss of
generality in assuming that $\varphi_j(0)=1$.

Let us now define the conformally flat $C^\om$ metrics
$g_j:=\varphi_jg_0$ on $\RR^n$. For notational simplicity we shall
denote by a subscript $j$ the geometric quantities corresponding to
the metric $g_j$, e.g.\ $\dd V_j$, $\nabla_j$ and $\De_j$. It is clear
that the end of $(\RR^n,g_j)$ is large because this manifold has
Euclidean volume growth. Moreover, the approximation~\eqref{Dal}
ensures that the Riemann tensor is bounded by
\[
\big|\Rm_j(x)\big|<C\,\e^{-|x|}\,,
\]
so that $(\RR^n,g_j)$ has asymptotically nonnegative curvature. By a
theorem of Kasue~\cite{Ka88}, this implies that $(\RR^n,g_j)$ has a
unique minimal positive Green's function $G_j:\RR^n\minus\{0\}\to\RR^+$,
which satisfies
\[
\De_j G_j=-\de_0\,,\qquad
\lim_{|x|\to\infty}G_j(x)=0\,.
\]

It is well known~\cite{SY95} that
\begin{equation}\label{conveq}
G_j(x)=\lim_{R\to\infty}G_j^R(x)
\end{equation}
for all $x\in\RR^n\minus\{0\}$, where $G_j^R\in C^\infty(\overline{A^R})$
stands for the unique solution to the boundary problem
\begin{subequations}\label{BC1R}
\begin{gather}
\De_j G_j^R=0\qquad \text{in
}A^R:=\big\{x\in\RR^n:\tfrac1R<|x|<R\big\}\,,\\
G_j^R\big|_{|x|=R^{-1}}=\frac1{|\SS^{n-1}|R^{n-2}}\,,\qquad G_j^R\big|_{|x|=R}=0\,.
\end{gather}
\end{subequations}
on the annulus of center $0$, inner radius $R^{-1}$ and outer
radius $R$. The limit~\eqref{conveq} is uniform on compact
subsets of $\RR^n\minus\{0\}$.

Now we shall fix any $R$ such that $B(0,R^{-1})\subset \Om\subset B(0,R)$ and show that
\begin{equation}\label{limGRj}
\lim_{j\to\infty}G^R_j(x)=0
\end{equation}
for all $x\in A^R\minus\overline \Om$, this limit being uniform on
compact subsets. In order to prove this, we shall use~\cite{GT98} that
$G_j^R$ is the unique minimizer of the functional
\[
E_j^R(F):=\int_{A^R}\big|\nabla_j
F\big|_{j}^2\,\dd V_j=\int_{A^R}\varphi_j(x)^{\frac n2-1}\big|\nabla F(x)\big|^2\,\dd x\,,
\]
defined on
\[
\mathcal C^R:=\big\{F\in
C^{0,1}(A^R):F\big|_{|x|=\frac1R}=\tfrac1{|\SS^{n-1}|R^{n-2}},\;F\big|_{|x|=R}=0\big\}\,.
\]

Clearly $\inf E_j^R$ is uniformly bounded for all $j$. For instance, if
$F_0$ belongs to the set
\[
\cC(\Om;R):=\big\{F\in C^{0,1}(A^R\cap
\Om):F\big|_{|x|=\frac1R}=\tfrac1{|\SS^{n-1}|R^{n-2}},\;F\big|_{\pd
\Om}=0\big\}\,,
\]
it immediately follows from the embedding
$\cC(\Om;R)\subset\cC^R$ that
\begin{equation}\label{bound}
\inf E_j^R\leq E_j^R(F_0)\leq\big(1+\tfrac1j\big)^{\frac n2-1}\int_{A^R\cap
\Om}|\nabla F_0|^2\,\dd x\,.
\end{equation}
In particular, by letting $F_0$ vary over $\cC(\Om;R)$ it stems that
\begin{equation}\label{upper}
\inf E_j^R\leq \big(1+\tfrac1j\big)^{\frac
n2-1}\inf\cE_\Om^R=\inf\cE_\Om^R+o(1)\,,
\end{equation}
where $\cE_\Om^R:\cC(\Om;R)\to\RR$ denotes the energy functional
\[
\cE_\Om^R(F):=\int_{A^R\cap \Om}|\nabla F|^2\,\dd x
\]
and the symbol $o(1)$ stands for a quantity that tends to zero as
$j\to\infty$.

Let us now suppose that there exist some $\ep_1,\ep_2>0$ and a sequence of nonnegative integers $(j_s)_{s=1}^\infty\nearrow\infty$ such that the Lebesgue measure of the set
\[
U_s^R:=\big\{x\in A^R\minus\overline \Om: \big|\nabla G_{j_s}^R(x)\big|\geq\ep_1\big\}
\]
is at least $\ep_2$ for all $s$. In this case
\[
\limsup_{j\to\infty}\inf E_j^R\geq\lim_{s\to\infty}\inf E_{j_s}^R=\lim_{s\to\infty}E_{j_s}^R(G_{j_s}^R)\geq \lim_{s\to\infty}(j_s-1)^{\frac n2-1}\ep_1^2\ep_2=\infty\,,
\]
contradicting the fact that $\inf E_j^R$ is bounded in
$j$. It then follows that $|\nabla G_j^R|$ tends to zero almost everywhere in
$A^R\minus\overline \Om$.

Let $K$ be any compact subset of $\overline{A^R}\minus\overline \Om$. Standard Schauder
estimates for the differential equation $\De_jG_j^R=0$ yield that~\cite{GT98}
\begin{equation}\label{boundj}
  \|G_j^R\|_{C^{2,\be}(K)} \leq C\| G_j^R\|_{C^0(A^R)}\leq \frac
C{|\SS^{n-1}|R^{n-2}}\,,
\end{equation}
where the maximum principle has been used to derive the second inequality. Moreover, the
constant $C$ can be chosen to depend on $K$, $\be$ and $n$ but not on
$j$, since $j\ms \De_jG_j^R=0$, the principal symbol
of the elliptic operator $j\ms \De_j$ is given by $j\,\vp_j^{-1}\,\id$ and
\[
\lim_{j\to\infty}j\,\vp_j^{-1}(x)=1
\]
for all $x\in K$.

$\nabla G_j^R|_K$ converging to zero almost everywhere, from Egorov's theorem and the uniform bound~\eqref{boundj} for the second derivative of $G_j^R$ it follows that $\nabla G_j^R|_K\to0$ uniformly as $j\to\infty$. The Dirichlet boundary condition then implies that $G_j^R$
tends to zero uniformly in $K$. As a consequence of this, one has that
\[
\lim_{j\to\infty}\big\|G_j^R\big\|_{H^1(K)}^2:=\lim_{j\to\infty}\bigg(\int_KG_j^R(x)^2\,\dd
x+\int_K\big|\nabla G_j^R(x)\big|^2\,\dd x\bigg)=0\,.
\]
In particular, one can find a sequence of smooth domains
$\Om_j\supset \overline \Om$ and functions $\hat
G_j^R\in\cC(\Om_j;R)\subset\cC^R$ such that
\[
\lim_{j\to\infty}\big\| G_j^R-\hat G_j^R\big\|_{H^1(A^R)}=0\,.
\]
Furthermore, the domains $\Om_j$ can be assumed to converge to $\Om$ in
the sense that there exists a sequence of ambient diffeomorphisms
$\Phi_j^R:\RR^n\to\RR^n$ mapping $(\Om,\pd\Om,\pd A^R)$ onto $(\Om_j,\pd\Om_j,\pd A^R)$ with $\|\Phi_j^R-\id\|_{C^1}\to0$ as $j\to\infty$.

Next we shall prove that
\begin{equation}\label{lower}
\inf E^R_j\geq\inf\cE^R_\Om+o(1)\,,
\end{equation}
which together with Eq.~\eqref{upper} implies that
\begin{equation}\label{infI}
\lim_{j\to\infty}\inf E^R_j=\inf\cE^R_\Om\,.
\end{equation}
In order to prove this claim, it suffices to note that
\begin{align}
\inf E^R_j&\geq\big(1-\tfrac1j)^{\frac n2-1}\int_{A^R}\big|\nabla G_j^R\big|^2\,\dd x\notag\\
&= \int_{A^R} \big|\nabla G_j^R\big|^2\,\dd x+o(1)\notag\\
&=\int_{A^R\cap \Om_j} \big|\nabla \hat G_j^R\big|^2\,\dd x+o(1)\notag\\
&=\int_{A^R\cap \Om} \big|\nabla \hat G_j^R\circ\Phi_j^R\big|^2\,\dd
x+o(1)\,.\label{GPhi}
\end{align}
The third equality follows directly from the definition of $\hat G_j^R$,
whereas the fourth one makes use of the fact that $\Phi_j^R$ maps $\Om$
diffeomorphically onto $\Om_j$ and $\|\Phi_j^R-{\rm id} \|_{C^1}\to0$. As
$\hat G_j^R\circ\Phi_j^R\in\cC(\Om;R)$, this yields Eq.~\eqref{lower} and hence~\eqref{infI}.

From Eqs.~\eqref{infI} and \eqref{GPhi} one immediately derives that
\begin{equation}
\lim_{j\to\infty} E_j^R(G_j^R)= \lim_{j\to\infty} \cE^R_\Om(\hat
G_j^R\circ\Phi_j^R)=\inf \cE^R_\Om\,,
\end{equation}
which is well known to imply~\cite{Sr00} that
\[
\lim_{j\to\infty}\big\|\hat G_j^R\circ\Phi_j^R-G^R_\Om\big\|_{H^1(A^R\cap
\Om)}=0\,,
\]
where $G_\Om^R\in C^\om(A^R\cap \Om)$ is the unique solution of
\begin{gather}\label{BC1R}
\De G_\Om^R=0\quad\text{in
}A^R\cap \Om\,,\qquad
G_\Om^R\big|_{|x|=R^{-1}}=\frac1{|\SS^{n-1}|R^{n-2}}\,,\qquad
G_\Om^R\big|_{\pd\Om}=0\,.
\end{gather}
The fact that $\big\|\Phi_j^R-\id \big\|_{C^1}\to0$ now implies that
$G_j^R$ also converges to $G^R_\Om$ in $H^1(A^R\cap \Om)$. Using a uniform gradient
bound for $|\nabla G_j^R|$ as in the proof of~\eqref{limGRj}, one finds that
\begin{equation}\label{lim2}
  \lim_{j\to\infty}G_j^R(x)=G_\Om^R(x)
\end{equation}
uniformly for $x$ in any compact subset of $A^R\cap \Om$.

It is a standard result that
\[
\lim_{R\to\infty} G_\Om^R(x)=G_\Om(x)\,,\qquad
\lim_{R\to\infty}G_j^R(x)=G_j(x)
\]
for all $x\in \Om\minus\{0\}$, this limit being uniform in compact
sets. Therefore, we conclude that $\lim_{j\to\infty}G_j(x)=G_\Om(x)$ uniformly for $x$ in any compact subset of $\Om\minus\{0\}$.

Let $K$ be a compact subset of $\Om\minus\{0\}$. Notice that $\varphi_j\De_ju=\De u+\lan Y_j,\nabla u\ran$, the vector field $Y_j:=(\frac n2-1)\vp_j^{-1}\nabla\vp_j$ satisfying
\begin{equation}\label{Xj}
  \| Y_j\|_{C^{k-2,\be}(K)}<\frac{C_{k,\be}}j
\end{equation}
for any $\be\leq1$ by~\eqref{Dal}. From the fact that
\[
\varphi_j\De_j(G_j-G_\Om)=-\lan Y_j,\nabla G_\Om\ran
\]
and standard Schauder estimates, one obtains that
\begin{align*}
  \|G_j-G_\Om\|_{C^{k,\be}(K)} &\leq C_1\big( \| G_j-G_\Om\|_{C^0(K)}+\|\nabla G_\Om\|_{C^{k-2,\be}(K)}\|Y_j\|_{C^{k-2,\be}(K)}\big)\,,
\end{align*}
where $C_1=C_1(K,k,\be)$ can be taken independent of $j$. From~\eqref{Xj} and the fact that $\|G_j-G_\Om\|_{C^0(K)}\to 0$ as $j\to\infty$, the theorem follows.
\end{proof}

The fact that Theorem~\ref{T:main} yields a $C^k$ approximation ($k\geq2$) is
crucial for relating the topological properties of the level sets of
the Green's function of the curved manifold to those of the Euclidean
Green's function of the domain. In what follows we shall present several
concrete applications of this idea as corollaries of the previous
theorem.  

\begin{corollary}\label{C:topo}
  Let us consider a $C^\infty$ compact domain $K\subset\Om$ and a small
  neighborhood $B$ of $y$. If $G_\Om$ is Morse, for any $\ep>0$ and $k\in\NN$ there exist
  another small neighborhood $B'$ of $y$, a $C^\infty$ diffeomorphism
  $\Theta:K\minus B\to K\minus B'$ with $\|\Theta-\id\|_{C^k(K)}<\ep$ and an analytic metric $g$ on
  $\RR^n$ whose minimal Green's function $G$ satisfies
\begin{equation}\label{diff}
G(x)=(G_\Om\circ\Theta)(x)
\end{equation}
for all $x\in K\minus B$. Moreover, $\Theta$ can be extended to
a $C^1$ diffeomorphism $K\to K$ so that Eq.~\eqref{diff} holds in
$K\minus\{y\}$.
\end{corollary}
\begin{proof}
By Theorem~\ref{T:main}, one can choose an analytic metric $g$ so that $G$ and $G_\Om$ are
arbitrarily close in the $C^{k+1}(K\minus B)$ topology. The structural stability of Morse
functions implies~\cite{AR67} that there exist a small
neighborhood $B'$ of $y$ and a $C^\infty$ diffeomorphism
$\Theta:K\minus B\to K\minus B'$ with $\|\Theta-\id\|_{C^k}<\ep$ and
satisfying~\eqref{diff} for all $x\in K\minus B$.

There is no loss of generality in assuming that $\pd B=G^{-1}(c)$ and $\pd B'=G_\Om^{-1}(c)$
with sufficiently large $c$. By Eq.~\eqref{asympG},
$f(x):=|x-y|^n\,G(x)$ and $f_\Om(x):=|x-y|^nG_\Om(x)$ are $C^2$ Morse functions in $B$ and $B'$, respectively. Define $B_1$ and $B_1'$ to be the bounded sets with
$\pd B_1=G^{-1}(c-1)$ and $\pd B'_1=G_\Om^{-1}(c-1)$. As $f$ and $f_\Om$ have a
minimum at $y$, there exists~\cite{Hi76} a $C^1$ diffeomorphism
$B_1\to B'_1$ mapping $G^{-1}(c')$ onto $G_\Om^{-1}(c')$ for all
$c'>c-1$. By a standard construction, this ensures that this
diffeomorphism can be chosen so as to match $C^1$ with $\Theta$ on
$B_1\minus\overline B$, yielding the desired diffeomorphism $K\to K$.
\end{proof}

\begin{corollary}\label{C:equip}
  Let $\Si\subset\RR^n$ be a compact, codimension $1$ submanifold
  without boundary of class $C^\infty$, and let $y$ be a point
  in the domain bounded by $\Si$. Then, for any $\ep>0$, there exist a compactly supported
  diffeomorphism $\Phi:\RR^n\to\RR^n$ with $\|\Phi-\id\|_{C^k}$ such that $\Phi(\Si)$ is a
  level set of the minimal Green's function with pole $y$ of a complete, analytic
  manifold~$(\RR^n,g)$.
\end{corollary}
\begin{proof}
  Let $\Om$ be the bounded domain enclosed by $\Si$ and let $G_\Om$ be
  its Green's function with pole $y$. By the Hopf boundary point
  lemma~\cite{GT98}, the gradient of $G_\Om$ does not vanish on
  $\Si$. The level sets of $G_\Om$ being connected, it then follows
  that the $C^\om$ submanifold $\Si':=G_\Om^{-1}(c)$ is
  $(\ep,k)$-close to $\Si$ for small enough $c>0$.

By Theorem~\ref{T:main} one can choose an analytic metric $g$ on
$\RR^n$ such that its minimal positive Green's function $G$ is arbitrarily close to $G_\Om$ in the
$C^k(\Om\minus\{y\})$ weak topology. The level sets of $G$ are necessarily compact because $G$ tends to zero at infinity. They are also connected: $\RR^n$ being contractible, if $G^{-1}(c)$ had more than one connected component, there would be at least two disjoint bounded sets $S_i$, $i=1,2$, such that $G|_{\pd S_i}$ would be constant. As the pole $y$ does not belong to one of these sets, say $S_1$, it follows from the maximum principle that $G|_{S_1}$ is constant, contradicting the unique continuation theorem. In turn, as it is connected, it follows by stability~\cite{AR67} that $G^{-1}(c)$ is diffeotopic to $\Si'$, so that $G^{-1}(c)=\Phi(\Si)$ for some ambient diffeomorphism close to the identity.
\end{proof}

Next we shall apply Theorem~\ref{T:main} to construct analytic
metrics in $\RR^n$ with an arbitrary number of critical points of
fixed Morse index. Let us recall that the {\em Morse index} of a nondegenerate critical point
$z$ of a $C^2$ function $f$ is the number of negative
eigenvalues of the Hessian matrix of $f$ at $z$.

\begin{corollary}\label{C:cp}
Let $\Om$ be a bounded domain in $\RR^n$ with $C^\infty$ connected
boundary and let $b_p(\overline\Om)$ be the Betti numbers of its
closure. Then there exists a $C^\om$
metric in $\RR^n$ whose minimal Green's function with pole $y$
has at least $b_p(\overline\Om)$ nondegenerate critical points of
Morse index $n-p$, for all $p=1,\dots,n-2$.
\end{corollary}
\begin{proof}
  By translating and slightly deforming the domain if necessary,
  Theorem~\ref{T:Morse} ensures that one can take $y\in\Om$ such that
  the Green's function $G_\Om$ of~$\Om$ with pole $y$ is Morse. By the
  Hopf boundary point lemma~\cite{GT98}, the gradient of $G_\Om$ does
  not vanish on the boundary of $\Om$.  Let us consider the
  $C^2(\Om)$ function given by $f:=-(G_\Om+1)^{-2}$ and denote by
  $\nu$ the exterior normal of $\pd\Om$.  Since $g(\nabla_g
  f(x),\nu(x))$ is negative for all $x\in\pd\Om$ and $y$ is an
  attractor of the local flow of $\nabla_gf$, by the
  Morse theory for manifolds with boundary~\cite{CM69} it follows that $f$ has at
  least $b_p(\overline{\Om})$ critical points of index $n-p$, for all $p=0,\dots,n$.
  Besides, all the critical points of $f$ other than $y$ are also critical points of
  $G_\Om$ and they have the same Morse indices. Therefore Corollary~\ref{C:topo} guarantees that there exists an analytic metric whose minimal Green's function has the same number of critical points in $\Om$ as $G_\Om$ and of the same Morse type. 
\end{proof}
\begin{remark}
For a manifold with boundary $\BOm$, $b_{n-1}(\BOm)=b_n(\BOm)=0$.
One also has that $b_0(\BOm)=1$, but the associated critical point
of $f$ simply reflects that $y$ is an attractor for the local
flow of $\nabla_g G_\Om$.
\end{remark}

To conclude the proof of Theorem~\ref{T:everything}, we shall next
show that the ideas developed in this section can be combined with those of Section~\ref{S:critical} to prove that there exists an analytic metric in $\RR^n$ such that $\Cr(G)$ has codimension at most $3$:

\begin{theorem}\label{T:codim3}
There exists an analytic metric in $\RR^n$ such that the critical
set of its minimal Green's function $G$ with pole $y$ has
codimension at most $3$.
\end{theorem}
\begin{proof}
First of all, let us take Cartesian coordinates $(x_1,\dots,x_n)$ in $\RR^n$ so that the pole of $G$ is given by $y=0$. Let us consider the polynomial
\[
Q(x_1,x_2):=\prod_{k=0}^{N-1}[(x_1-2k-1)^2+x_2^2-1\big]^2\,,
\]
where $N$ is a positive integer, and define the domain in $\RR^3$
\begin{align}\notag
\Om_0&:=\big\{x\in\RR^3: Q(x_1,x_2)+x_3^2< a\,,\; x_1>0\big\}\,.
\end{align}
$\Om_0$ is obviously diffeomorphic to a solid torus of genus $N$ in $\RR^3$ if
$a>0$ is sufficiently small.

By Corollary~\ref{C:cp} it suffices to consider the case where
$n\geq4$. We embed $\RR^3$ into $\RR^n=\RR^3\times\RR^{n-3}$ via
the map $\imath(x):=(x,0)$. Let $H\subset\mathrm{SO}(n)$ be the
group of rotations generated by the vector fields
$
X_j:=x_1e_j-x_je_1
$,
where $j=4,\dots, n$ and $e_k$ stands for
the $k$-th unit vector of the canonical basis. We define a domain
in $\RR^n$ as
\begin{align*}
\Om_1&:=\inte\big(H\cdot\imath(\overline{\Om_0})\big)=\inte\big\{R\circ\imath(x):x\in
\overline{\Om_0},\;R\in H\big\}\,.
\end{align*}

The boundary of the domain $\Om_1$ has corners on $\Pi\cap
\pd\Om_1$, where $\Pi$ denotes the 2-plane
$\{x\in\RR^n:x_1=x_4=\cdots=x_n=0\}$, so we start by defining a
$C^2$ domain $\Om$ by rounding off the corners of $\Om_1$. This new domain can be taken arbitrarily $C^0$-close to $\Om_1$. Moreover, as $\Om_1$ is
invariant under $H$ and under the involutions
$\si_k:\RR^n\to\RR^n$ given by
\[
\si_k(x):=x-2\lan x,e_k\ran\, e_k\qquad (k=1,\dots,n)\,,
\]
we can take $\Om$ to be invariant under these
transformations as well. We  call $H_0$ the group generated by
$\si_2$ and $\si_3$, which is isomorphic to $\ZZ_2\oplus\ZZ_2$, and set
\[
L:=\big\{x\in\RR^n:\, x_2=\cdots=x_n=0\big\}\,.
\]

We denote by $G_{\Om}$ the Euclidean Dirichlet Green's
function of $\Om$ with a pole at $0$. $G_{\Om}$ is
invariant under the Euclidean isometries which preserve both the
pole $y=0$ and the domain $\Om$, so $G_{\Om}\circ
\si_k=G_{\Om}$ for $k=1,\dots,n$. In particular, the subset
$(L\cap\Om)\minus\{0\}$ is therefore invariant under the local
flow of the gradient field $\nabla G_{\Om}$.

Let us consider the inclusion $\psi_\Om:L\cap\Om
\to\Om$. Clearly $L\cap \Om$ has $2N+1$ connected components, which we shall
denote by $L_\al$ ($\al=-N,\dots, N$). Let us order these components, by relabeling them if necessary, so that $\lan
e_1,x\ran<\lan e_1,x'\ran$ if $x\in L_\al$, $x'\in L_{\al'}$ and
$\al<\al'$. The function
\[
\Psi_\Om:=G_{\Om}\circ \psi_\Om:(L\cap\Om)\minus\{0\}\to\RR
\]
is of class $C^\om$ in $L_\al$ and vanishes on $\pd L_\al$ by the
boundary conditions of $G_{\Om}$. Moreover, $\Psi_\Om$ is
everywhere positive because so is $G_{\Om}$, which implies that $\Psi_\Om$ has a local maximum in $L_\al$ for
all $\al\neq0$  by Rolle's theorem. The invariance of $(L\cap\Om)\minus\{0\}$ under
the local flow of $\nabla G_{\Om}$ ensures that the aforementioned
maxima of $\Psi_\Om$ correspond to critical points of $G_{\Om}$, which are necessarily
of saddle type.

Let $z\in L\cap\Om$ be one of the above critical points of $G_{\Om}$. As
both $G_{\Om}$ and $\nabla G_\Om$ are $H$-equivariant, $\Cr(G_\Om)$ must
contain the $H$-orbit passing through $z$, which has dimension $n-3$. Thus
the critical set of $G_\Om$ has codimension at most $3$.

We shall now construct a sequence of analytic metrics $g_j$ on $\RR^n$
such that the minimal Green's functions $G_j$ in $(\RR^n,g_j)$ with pole $0$
approximate $G_\Om$ weakly in $\Om\minus\{0\}$. The construction is based
on a modification of the methods of proof of the Theorem~\ref{T:main}
adapted to the symmetries of the domain $\Om$. Again we denote by
$\tilde\vp_j:\RR^n\to[1,\infty)$ a smooth function such that
$\tilde\varphi_j(x)=1$ if $x\in \Om$ and $\tilde\varphi_j(x)=j$ if
$\dist(x,\Om)>\frac1j$. $\Om$ being ($H\oplus H_0$)-invariant, we can obviously take $\tilde\varphi_j$ invariant under $H\oplus H_0$. For any $j\in\NN$ there exists an analytic function
$\hat\varphi_j:\RR^n\to\RR$ such that~\cite[Section 1.6.5]{Na68}
\begin{equation}\label{Dal2}
\sum_{|\al|\leq3}|D^\al\hat\varphi_j(x)-D^\al\tilde\varphi_j(x)|<\frac{\e^{-|x|}}j\,.
\end{equation}
We can obviously assume that $\hat\varphi_j(0)=1$.

Observe that both the set
\[
\La_j:=\big\{x\in\RR^n:\dist(x,\pd\Om)>\tfrac1j\big\}
\]
and the majorating function $\e^{-|x|}/j$ in~\eqref{Dal2} are
invariant under the compact group $H\oplus H_0$. Consequently, we can
define a symmetrization $\varphi_j\in C^\om(\RR^n)$ of the function
$\hat\varphi_j$ by
\[
\varphi_j(x):=\frac14\sum_{\si\in H_0}\int_H \hat\varphi_j(\si\circ
h(x))\,\dd h\,,
\]
where $\dd h$ denotes the normalized Haar measure of $H$.  Since
$D^\al\tilde\varphi_j|_{\La_j}=0$, it readily follows from
Eq.~\eqref{Dal2} that for all $x\in\RR^n$ with
$\dist(x,\pd\Om)>\frac1j$ we have the estimate:
\begin{align*}
\sum_{|\al|\leq3}|D^\al\varphi_j(x)-D^\al\tilde\varphi_j(x)|&\leq
\frac14\sum_{\si\in H_0}\int_H
\sum_{|\al|\leq3}\big|D^\al(\hat\varphi_j\circ\si\circ
h)(x)-D^\al\tilde\varphi_j(x)\big|\,\dd h\\
&\leq\frac{C\e^{-|x|}}j\,,
\end{align*}
where $C$ does not depend on $j$.

If we now define the complete metric $g_j:=\varphi_jg_0$, Kasue's
theorem~\cite{Ka88} shows that $(\RR^n,g_j)$ admits a minimal
positive Green's function $G_j$, and from the proof of Theorem~\ref{T:main} it follows that
$G_j$ approximates $G_\Om$ weakly in $C^l(\Om\minus\{0\})$ for any $l\in\NN$.
Moreover, $\si_k$ is an isometry of $(\RR^n,g_j)$ by construction, so that the line
$L$ must be invariant under the local flow of the gradient field
$\nabla_j G_j$. As in the proof of Theorem~\ref{T:main}, we are denoting by a subscript $j$ the objects corresponding to the metric $g_j$.

Let us use the notation $\psi:L\to\RR^n$ for the inclusion and
set $\Psi_j:=G_j\circ\psi\in C^\om(L\minus\{0\})$. As
$\Psi_j\to\Psi_\Om$ uniformly on compact subsets of
$(L\cap\Om)\minus\{0\}$ and $\Psi_\Om$ has at least $2N$ local maxima,
it is standard that $\Psi_j$ also has at least $2N$ local maxima if
$j$ is large enough. Rolle's theorem and the fact that $\Psi_j$ tends
to $+\infty$ at $0$ ensure that $\Psi_j$ also has at least $2N$ local minima. By symmetry, these
local extrema of $\Psi_j$ correspond to critical points of $G_j$, and the
invariance of $G_j$ under $H$ implies that $\Cr(G_j)$ also
contains the $H$-orbit passing through each of these critical points,
which has dimension $n-3$.  This shows that the critical set of $G_j$ has
codimension at most $3$ for large $j$, as claimed.
\end{proof}
\begin{remark}
The proof of Theorem~\ref{T:codim3} relies on the construction of a
metric in $\RR^n$ with an $\SO(n-2)$ isometry subgroup leaving a point
$y$ invariant and whose Green's function $G$ with pole $y$ has a
nonempty critical set. However, we saw in
Theorem~\ref{P:absence} that the existence of an $\SO(n-1)$
isometry group automatically implies that $\Cr(G)=\emptyset$, so this construction cannot be adapted to obtain critical sets of
codimension~2. The
question of whether $\dim\Cr(G)\leq n-3$ for any Li--Tam Green's
function of a contractible manifold remains open.
\end{remark}

\section*{Acknowledgements}

The authors are indebted to Pawel Goldstein and Tadeusz Mostowski for valuable discussions.  A.E.\ is partially supported by the DGI and the Complutense University--CAM under grants no.~FIS2008-00209
and~GR58/08-910556. D.P.-S. acknowledges the financial support of the Spanish MICINN through the Ram\'on y Cajal program and
the partial support of the DGI under grant no.~MTM2007-62478.

\bibliographystyle{amsplain}

\end{document}